\documentclass[reqno]{amsart}
\usepackage{amsmath}
\usepackage{amscd,amsfonts,amssymb,amsthm,cite,mathrsfs,color}
\usepackage{dsfont} \usepackage{ leftidx}

\newtheorem{theorem}{Theorem}[section]

\newtheorem{proposition}[theorem]{Proposition}
\newtheorem{coro}[theorem]{Corollary}
\theoremstyle{definition}
\newtheorem {definition}[theorem]{Definition}
\newtheorem{remark}[theorem]{Remark}

\newtheorem*{acknow}{Acknowledgments}

\numberwithin{equation}{section}
\numberwithin{theorem}{section}

\def\be{\begin{equation}}
\def\ee{\end{equation}}
\def\bae{\begin{eqnarray}}
\def\eae{\end{eqnarray}}

\def\E{\mathbb{E}}

\begin{document}

\title[Limits for circular Jacobi ensembles]{Limits for circular Jacobi  beta-ensembles}

\author{Dang-Zheng Liu}  \address{School of Mathematical Sciences, University of Science and Technology of China, Hefei 230026, P.R. China \&\
Wu Wen-Tsun Key Laboratory of Mathematics, USTC, Chinese Academy of Sciences, Hefei 230026, P.R. China}
\email{dzliu@ustc.edu.cn}

 \date{\today}


\begin{abstract}
Bourgade, Nikeghbali and Rouault recently proposed a matrix model for the circular Jacobi $\beta$-ensemble, which is a generalization of the Dyson circular $\beta$-ensemble but   
equipped with an additional parameter $b$, and further studied its limiting spectral measure.   We calculate the scaling  limits for  expected products  of 
characteristic polynomials of circular Jacobi $\beta$-ensembles. For the fixed constant $b$, the resulting limit near the spectrum singularity
is proven to be a new   multivariate   function. When  $b=\beta Nd/2$, the scaling limits in the bulk and at the soft edge 
agree with those of the Hermite (Gaussian), Laguerre (Chiral) and Jacobi $\beta$-ensembles proved in \cite{dl}.  As corollaries, for even $\beta$  the scaling  limits  
of point correlation functions for the  ensemble  are given. Besides, a transition  from  the spectrum singularity  to the soft edge limit  is observed as $b$ goes to infinity.
The positivity of two special multivariate hypergeometric functions, which appear as one factor   of the joint eigenvalue densities for spiked Jacobi/Wishart  $\beta$-ensembles and Gaussian   $\beta$-ensembles with source, will also be shown.
\end{abstract}

\maketitle

\section{Introduction}

\subsection{Circular Jacobi ensembles}

The circular unitary ensemble in random matrix theory usually refers to the unitary group $U(N)$ with its invariant Haar measure. The  correlation functions and   moments of characteristic polynomials  for eigenvalues of these random unitary matrices are predicted to be somehow  related with statistical properties of the Riemann zeta function in number theory \cite{ks}. Together with circular orthogonal and symplectic ensembles (see  \cite[Chap. 2]{forrester} for definition and   relationship to compact symmetric spaces) they consist Dyson three-fold way for circular ensembles. Dyson also observed that the induced eigenvalue densities correspond to the Gibbs measure of Coulomb log-gases on the circle at three different inverse temperatures $\beta=1, 2, 4$. For general $\beta>0$  these are called $\beta$-ensembles in the literature.

Inspired by Dumitriu-Edelman tridiagonal matrix models for Gaussian and Laguerre $\beta$-ensembles \cite{due}, Killip and Nenciu constructed   analog matrix models for Dyson circular $\beta$-ensembles and also for Jacobi $\beta$-ensembles  \cite{KN}. Combining circular and Jacobi  ensembles, a more general ensemble  was  defined as the circular Jacobi ensemble  in \cite{forrester,nagao}.  Explicitly, the circular Jacobi ensemble refers to the probability density function  on $[0,2\pi)^{N}$
\begin{equation}
P_{b,N}(\theta_{1},\ldots,\theta_{N})=\frac{1}{M_{N}(\bar{b},b;\beta/2)}\prod_{j=1}^{N}\left( e^{i\frac{\bar{b}-b}{2} (\theta_{j}-\pi) }|1-e^{i\theta_{j}}|^{\bar{b}+b}\right) \ |\Delta_{N}(e^{i\theta})|^{\beta}\label{cjdensity}
\end{equation}
with $\textrm{Re}\{b\}>-1/2$, see \cite[Chap. 3]{forrester} for detailed discussion and  \cite{bnr} for the matrix model realization.
Here the   constant $M_{N}(a', b';\alpha)$    equals to (see, e.g., \cite[Chap. 4]{forrester})
\be \label{Mconstant}M_{N}( a', b';\alpha)= (2\pi)^{N}\prod_{j=0}^{N-1}
\frac{\Gamma(1+ \alpha+j \alpha) \Gamma(1+a'+b'+j \alpha)}
{\Gamma(1+ \alpha)\Gamma(1+a'+j \alpha)\Gamma(1+b'+j \alpha)}.
 \ee
 and  the Vandermande determinant  $$ \Delta_{N}(e^{i\theta})=\prod_{1\leq j<k\leq N}(e^{i\theta_{k}}-e^{i\theta_{j}}).$$

For $\beta=2$, such measures were  first  studied by Hua \cite{hua} and later by Nagao \cite{nagao} motivated by random matrix theory. For $\beta=1, 2, 4$ the point correlation functions and their scaling limits were
investigated by Forrester and Nagao \cite{fn01}. Note that for $b=0$ \eqref{cjdensity} reduces to the density for the circular ensemble and for $b \in \tfrac{\beta}{2}\mathbb{N}$ it coincides with the circular ensemble given that there is an eigenvalue at 1. For general $b$ it is said to describe
a spectrum singularity.

Applying an appropriate change of variables we can derive some classical random matrix ensembles from \eqref{cjdensity}. For instance, with
$$e^{i\theta_{j}}=\frac{1+i\lambda_{j}}{1-i\lambda_{j}}, \qquad j=1, 2, \ldots, N,$$  \eqref{cjdensity} reduces to the
Cauchy ensemble, which shows that both ensembles are in fact equivalent, see Chapter 3.9, \cite{forrester}. Another example is that
the Laguerre ensemble  can also be treated as a limit transition from \eqref{cjdensity}.
The latter fact is a remark of E.M. Rains (AIM workshop 2009) communicated to the author by P.J. Forrester, which can be formally inferred as follows.
Let $c=\bar{b}+b$ be a fixed constant and let  $\textrm{Im}\{b\}=i\frac{\bar{b}-b}{2}$ go to $-\infty$, with change of variables
$$\theta_j=-\frac{x_j}{\textrm{Im}\{b\}}, \qquad j=1, 2, \ldots, N,$$ we see that
\begin{align}P_{b,N}(\theta)d^{N}\theta  \sim
 \frac{1}{Z_N} \prod_{j=1}^{N} x_{j}^{c} e^{-x_j}  \ |\Delta_{N}(x)|^{\beta}d^{N}x.
\end{align}
Here all $x_j\in (0,\infty)$ and  \be Z_N= M_{N}(\bar{b},b;\beta/2)   e^{\pi N\textrm{Im}\{b\}}(-\textrm{Im}\{b\})^{(c+1)N+\beta N(N-1)/2}\ee
goes to the normalization constant for the Laguerre ensemble (see (3.129) in \cite{forrester} for the explicit constant) as $\textrm{Im}\{b\}\rightarrow -\infty$ with the help of Stirling formulas.

\subsection{Limiting  spectral measure} The empirical spectral distribution  for the circular Jacobi ensemble refers to
\be \mu_N=\frac{1}{N}\sum_{k=1}^{N}\delta_{e^{i\theta_k}}.\ee
With $b=b_N=\beta N\textrm{d}/2$ where $\textrm{Re}\{\textrm{d}\}\geq 0$, as $N\rightarrow \infty$   Bourgade etc. \cite{bnr} have proved that  $\mu_N$ converges weakly in probability to some  limit $\mu_{\textrm{d}}(\zeta)$ ($\zeta=e^{i\theta}$).   Jokela etc.  \cite{jok}  gave  another alternative derivation
from the viewpoint of  the equilibrium measure.   To give the explicit form of $\mu_{\textrm{d}}(\zeta)$, we need some more notations.

Let $\theta_\textrm{\textrm{d}}\in [0,\pi)$ and $\xi_\textrm{d} \in [-\theta_\textrm{d},\theta_\textrm{d}]$ be such that
\be \sin\frac{\theta_\textrm{d}}{2}=\left|\frac{\textrm{d}}{1+\textrm{d}}\right|, \qquad e^{i \xi_\textrm{d}}=\frac{1+\textrm{d}}{1+\bar{\textrm{d}}}.\ee
If $\textrm{Re}\{\textrm{d}\}\geq 0, \textrm{d}\neq 0$, let $\omega_{\textrm{d}}$ be defined by
\be \omega_{\textrm{d}}(\theta)=\begin{cases}  \frac{\sqrt{\sin^{2}((\theta-\xi_\textrm{d})/2)-\sin^{2}(\theta_\textrm{d}/2)}}{|1/(1+\bar{\textrm{d}})|\sin(\theta/2)}, & \theta \in (\theta_\textrm{d}+\xi_\textrm{d}, 2\pi-\theta_\textrm{d}+\xi_\textrm{d}) \\ 0, & \textrm{otherwise}. \end{cases} \ee
In this case, let \be d\mu_{\textrm{d}}(\zeta):=\omega_{\textrm{d}}(\theta)\frac{\textrm{d}\theta}{2\pi} \ (\zeta=e^{i\theta}),\ee
and for $\textrm{d}=0$, let
\be \textrm{d}\mu_{0}(\zeta):= \frac{\textrm{d}\theta}{2\pi}\  (\zeta=e^{i\theta}), \ee
the Haar measure on the unit circle.

In particular, if $\textrm{d}\in \mathbb{R}$ and $\textrm{d}\geq 0$, then $\xi_\textrm{d}=0$. For  $\textrm{d}>0$, $\theta=\theta_\textrm{d}$ (or $2\pi-\theta_\textrm{d}$) is the so-called soft edge of the spectrum  while the interval $(\theta_\textrm{d}, 2\pi-\theta_\textrm{d})$ is the bulk of the spectrum. In this case,  we will investigate the local limits for correlations  of characteristic polynomials in the circular Jacobi ensembles at the edge
and in the bulk as in the Hermite (Gaussian), Laguerre (Chiral) and Jacobi $\beta$-ensembles of \cite{dl}.

\subsection{Correlations
}

For the characteristic polynomial of the circular Jacobi ensemble denoted by \be \phi(z)=\prod_{j=1}^{N}(1-z e^{i\theta_{j}}) \ee and its conjugate by
\be \overline{\phi(z)}=\prod_{j=1}^{N}(1-\bar{z} e^{-i\theta_{j}}), \ee
our aim is to prove local limits of the expectation of products of characteristic polynomials
\begin{equation}\label{productmean}
K_{b,N}(s;t)=\E[\, \prod_{j=1}^m \phi(s_{j}) \prod_{k=1}^n \overline{\phi(t_{k})}\, ],
\end{equation}
as $N\rightarrow \infty$. Here $\E[f]$ means the expectation of $f$ under the circular Jacobi ensemble.

At this point, it is worth stressing  that the above average $K_{b,N}(s;t)$ allows the evaluation of $k$-point correlation functions for even $\beta$.
Suppose for convenience that there are $k+N$ eigenvalues denoted by $e^{ir_1},\ldots, e^{ir_k}, e^{i\theta_1}, \ldots, e^{i\theta_N}$ in the circular Jacobi
ensemble,  or equivalently the density function of \eqref{cjdensity} is replaced by $P_{b, k+N}(r_{1},\ldots,r_{k},\theta_{1},\ldots,\theta_{N})$.  Accordingly, the $k$-point correlation function \cite{forrester,mehta} is specified as a canonical average by
\begin{multline} \label{defpcf}R_{b,N}^{(k)}(r_1,\ldots,r_k)= \frac{(k+N)!}{N!} \frac{1}{M_{k+N}(\bar{b},b;\beta/2)}
|\Delta_{k}(e^{ir})|^{\beta}\prod_{j=1}^{k} e^{i\frac{\bar{b}-b}{2} (r_{j}-\pi) }|1-e^{ir_{j}}|^{\bar{b}+b} \\
\times \int  \prod_{l=1}^{k}\prod_{j=1}^{N}  |e^{ir_{l}}-e^{i\theta_{j}}|^{\beta}
\prod_{j=1}^{N} e^{i\frac{\bar{b}-b}{2} (\theta_{j}-\pi) } |1-e^{i\theta_{j}}|^{\bar{b}+b} \ |\Delta_{N}(e^{i\theta})|^{\beta} d^{N} \theta.
\end{multline}
Therefore for even $\beta$, if we let $m=\beta k/2$ and introduce\be \label{variablescptopcf} s_{l+ \beta(j-1)/2}=e^{-ir_j}, \quad j=1,\ldots,k, l=1,\ldots, \beta/2,\ee
 then   it is easy to verify that
\begin{align} R_{b,N}^{(k)}(r_1,\ldots,r_k)&= \frac{(k+N)!}{N!} \frac{ M_{N}(\bar{b},b;\beta/2)}{M_{k+N}(\bar{b},b;\beta/2)}
|\Delta_{k}(e^{ir})|^{\beta} \nonumber\\
&\times \prod_{j=1}^{k} e^{i\frac{\bar{b}-b}{2} (r_{j}-\pi) }|1-e^{ir_{j}}|^{\bar{b}+b}\ K_{b,N}(s;s)\label{cptopcf}.
\end{align}

\subsection{Goal and organization of the paper}

We  consider the expectation of products of characteristic polynomials \eqref{productmean} for the circular Jacobi $\beta$-ensemble defined by \eqref{cjdensity}. Our aim is to prove the   local limit for the  expectation as $N\rightarrow \infty$ and to give  explicit expression of the limit function in terms of Jack polynomials  for any $\beta>0$.  As a direct application, for even $\beta$  we   prove the local limit of  correlation functions.  The similar results have been shown for other  three classical $\beta$-ensembles in random matrix theory, namely the  Hermite (Gaussian), Laguerre (Chiral), and  Jacobi $\beta$-ensembles \cite{dl}.

For special values of the Dyson index $\beta=1, 2, 4$, we recover more conventional random matrix ensembles, see \cite{forrester,mehta}.  Indeed,
for   $\beta=1,2,4$,  the scaling limits for  ratios and products of characteristic polynomials  are already well established, see \cite{af,bds,bhn08,bs,hko01} and references therein.

In the present article we deal with   the  general $\beta$ case. For the Gaussian $\beta$-ensembles,  Aomoto \cite{aom2} and more recently Su \cite{su} proved the scaling limit for the product of $n=2$ characteristic polynomials. In the previous paper \cite{dl}, we (with P. Desrosiers) completely proved the scaling limit for
products of arbitrary $n$ characteristic polynomials in the  Hermite (Gaussian), Laguerre (Chiral), and  Jacobi $\beta$-ensembles.  Other  properties concerning   $\beta$-ensembles, for instance,
duality and Jack polynomials\cite{bf,des,df,dl14,duk,forrester0,mat},  stochastic differential equations   \cite{es,rr,rrv},   the global fluctuations \cite{bg13,due2,due3,kill}, the gap probabilities  and the
largest eigenvalues \cite{dv,hmf,jiang,rrv,rrz,vv,vv2},   were also studied in the limit $N\to \infty$.  More recently, some universality results  concerning the general $\beta>0$ and general potential case have been obtained \cite{bey,bey2}. Much progress has been made for the circular Jacobi ensembles, see \cite{bnr13,cnn14,fn01,mnn13,nnr13,su09}.

The rest of this paper is organized as follows. In Section 2 we briefly review some aspects of Jack polynomials and associated hypergeometric functions. The positivity of two special functions ${\phantom{j}}_1\mathcal{F}^{(\alpha)}_0$ and ${\phantom{j}}_0\mathcal{F}^{(\alpha)}_0$ will be shown.
In Sections 3 and 4 we turn to calculations of the local limits of the bulk and edge respectively for the fixed constant $b$  and  for $b=\beta Nd/2$.
In particular for the former a new multivariate function is found to appear as a scaling limit near the spectrum singularity. In Section 5 we derive a limit transition from the spectrum singulary to the soft edge.

\section{Jack polynomials and hypergeometric functions}
\label{hypergeometric}
This section   provides a brief review of  Jack polynomials  and associated hypergeometric functions \cite{kaneko,stanley,yan}, see \cite[Chap. 12]{forrester} for more recent textbook treatments. 

\subsection{Partitions}

A partition $\kappa = (\kappa_1, \kappa_2,\ldots,\kappa_i,\ldots,)$ is a sequence of non-negative integers $\kappa_i$ such that
\begin{equation*}
    \kappa_1\geq\kappa_2\geq\cdots\geq\kappa_i\geq\cdots
\end{equation*}
and only a finite number of the terms $\kappa_i$ are non-zero. The number of non-zero terms is referred to as the length of $\kappa$, and is denoted $\ell(\kappa)$. We shall not distinguish between two partitions that differ only by a string of zeros. The weight of a partition $\kappa$ is the sum
\begin{equation*}
    |\kappa|:= \kappa_1+\kappa_2+\cdots
\end{equation*}
of its parts, and its diagram is the set of points $(i,j)\in\mathbb{N}^2$ such that $1\leq j\leq\kappa_i$.
Reflection in the diagonal produces the conjugate partition
$\kappa^\prime=(\kappa_1',\kappa_2',\ldots)$.

The set of all partitions of a given weight is  partially ordered
by the dominance order: $\kappa\leq \sigma $ if and only if $\sum_{i=1}^k\kappa_i\leq \sum_{i=1}^k \sigma_i$ for all $k$.

\subsection{Jack  polynomials}

Let $\Lambda_n(x)$ be the algebra of symmetric polynomials in $n$ variables $x_1,\ldots,x_n$ with coefficients in the field $\mathbb{F}=\mathbb{Q}(\alpha)$, which is  the field of  rational functions in the parameter $\alpha$.  As a ring, $\Lambda_n(x)$ is generated by the power-sums:
\be \label{powersums} p_k(x):=x_1^k+\cdots+x_n^k. \ee
The ring of symmetric polynomials is naturally graded: $\Lambda_n(x)=\oplus_{k\geq 0}\Lambda^k_n(x)$, where $\Lambda^k_n(x)$ denotes the set of homogeneous polynomials of degree $k$.   As a vector space, $\Lambda^{k}_n(x)$   equals to the span over  $\mathbb{F}$ of all symmetric monomials $m_\kappa(x)$, where $\kappa$ is a partition of weight $k$ and
\be  m_\kappa(x):=x_1^{\kappa_1}\cdots x_n^{\kappa_n}+\text{distinct permutations}.\nonumber
\ee
Note that if the length of the partition $\kappa$ is larger than $n$,  we set $m_\kappa(x)=0$.

The whole ring $\Lambda_n(x)$ is invariant under the action of  homogeneous differential operators related to the Calogero-Sutherland models \cite{bf}:
\be E_k=\sum_{i=1}^n x_i^k\frac{\partial}{\partial x_i}, \ D_k=\sum_{i=1}^n x_i^k\frac{\partial^2}{\partial x_i^2}+\frac{2}{\alpha}\sum_{1\leq i\neq j \leq n}\frac{x_i^k}{x_i-x_j}\frac{\partial}{\partial x_i},\  k=0,1,2,\ldots .
\ee
The operators $E_1$ and $D_2$   can be used to define the Jack polynomials.  Indeed, for each partition $\kappa$, there exists a unique symmetric polynomial $P^{(\alpha)}_\kappa(x)$ that satisfies the following two  conditions \cite{stanley}:
\begin{align} \label{Jacktriang}(1)\qquad&P^{(\alpha)}_\kappa(x)=m_\kappa(x)+\sum_{\mu<\kappa}c_{\kappa\mu}m_\mu(x)&\text{(triangularity)}\\
\label{Jackeigen}(2)\qquad &\Big(D_2-\frac{2}{\alpha}(n-1)E_1\Big)P^{(\alpha)}_\kappa(x)=\epsilon_\kappa P^{(\alpha)}_\kappa(x)&\text{(eigenfunction)}\end{align}
where  $\epsilon_\kappa, c_{\kappa\mu} \in \mathbb{F}$.  Because of the tiangularity condition, $\Lambda_n(x)$ is also equal to the span over $\mathbb{F}$ of all Jack polynomials  $P^{(\alpha)}_\kappa(x)$, with $|\kappa|\leq n$.

\subsection{Hypergeometric series}\label{subsecseries}

For    $(i,j) \in \kappa$, let
\begin{equation}\label{lengths}
    a_\kappa(i,j) = \kappa_i-j\qquad\text{and}\qquad l_\kappa(i,j) = \kappa^\prime_j-i.
\end{equation}Introduce the hook-length of $\kappa$    defined by
\begin{equation} \label{defhook}
    h_{\kappa}^{(\alpha)}=\prod_{(i,j)\in\kappa}\Big(1+a_\kappa(i,j)+\tfrac{1}{\alpha}l_\kappa(i,j)\Big),
\end{equation}
and  the   $\alpha$-deformation of the Pochhammer symbol by
\begin{equation}\label{defpochhammer}
    [x]^{(\alpha)}_\kappa = \prod_{1\leq i\leq \ell(\kappa)}\Big(x-\tfrac{i-1}{\alpha}\Big)_{\kappa_i}.
\end{equation}
Here  $(x)_j\equiv x(x+1)\cdots(x+j-1)$.

With the  previous notations  we   now turn  to   the  precise definition of the hypergeometric series associated with Jack polynomials, see e.g. \cite{kaneko,yan}.
 Fix $p,q\in\mathbb{N}_0=\{0, 1, 2,\ldots\}$ and let $a_1,\ldots,a_p, b_1,\ldots,b_q$ be complex numbers such that $(i-1)/\alpha-b_j\notin\mathbb{N}_0$ for all $i\in\mathbb{N}_0$.
The $(p,q)$-type hypergeometric series in two sets of $n$ variables  $x=(x_1,\ldots,x_n)$ and  $y=(y_1,\ldots,y_n)$  reads off
\begin{align}\label{hfpq2}
    {\phantom{j}}_p\mathcal{F}^{(\alpha)}_{q}&(a_1,\ldots,a_p;b_1,\ldots,b_q;x;y)
                = \nonumber\\
             & \sum_{k=0}^{\infty}\, \sum_{\ell(\kappa)\leq n,|\kappa|=k} \frac{1}{h_{\kappa}^{(\alpha)}}\frac{\lbrack a_1\rbrack^{(\alpha)}_\kappa\cdots\lbrack a_p\rbrack^{(\alpha)}_\kappa}{\lbrack b_1\rbrack^{(\alpha)}_\kappa
    \cdots\lbrack b_q\rbrack^{(\alpha)}_\kappa}\frac{P_{\kappa}^{(\alpha)}(x)P_{\kappa}^{(\alpha)}(y)}{P_{\kappa}^{(\alpha)}(1^{(n)})},
\end{align}
where  the shorthand notation $1^{(n)}$ stands for $ 1,\ldots,1 $  with  $n$ times and
\be P_{\kappa}^{(\alpha)}(1^{(n)})=\prod_{(i,j)\in\kappa}\frac{n+\alpha (j-1)-(i-1) }{1+\alpha a_\kappa(i,j)+ l_\kappa(i,j)}
,\label{Pvalue}\ee
see (10.20),\cite{macdonald}.
In particular, when   $y_1=\cdots=y_n=1$, it reduces to the hypergeometric series
\begin{equation}\label{hfpq}
  {\phantom{j}}_p F^{(\alpha)}_{q}(a_1,\ldots,a_p;b_1,\ldots,b_q;x) = \sum_{k=0}^{\infty}\sum_{|\kappa|=k} \frac{1}{h_{\kappa}^{(\alpha)}}\frac{\lbrack a_1\rbrack^{(\alpha)}_\kappa\cdots\lbrack a_p\rbrack^{(\alpha)}_\kappa}{\lbrack b_1\rbrack^{(\alpha)}_\kappa
    \cdots\lbrack b_q\rbrack^{(\alpha)}_\kappa}P_{\kappa}^{(\alpha)}(x).
\end{equation}

From now on we always assume $\alpha>0$.  Note that when $p\leq q$, the   series \eqref{hfpq}  converges absolutely for all $x\in\mathbb{C}^n$,  while for $p=q+1$,  \eqref{hfpq} converges absolutely for   all $ | x_i|<r$ with some positive constant $r$, see \cite[Proposition 1]{kaneko} for more details about the convergence issue. It immediately follows from the definition and the convergence of \eqref{hfpq} that  for $p\leq q$  \eqref{hfpq2}    converges absolutely for all $x, y\in\mathbb{C}^n$, while for $p=q+1$  \eqref{hfpq2} converges absolutely for   all $ | x_i y_j|<r$.

In the special case of $p=1, q=0$ we can say more about the convergence since
\begin{equation}\label{hf10}
  {\phantom{j}}_1 F^{(\alpha)}_{0}(a;x) = \prod_{j=1}^{n}(1-x_j)^{-a},
\end{equation}
 see \cite[Proposition 3.1]{yan}. Therefore, the series${\phantom{j}}_1 F^{(\alpha)}_{0}(a;x)$ converges for   all $ | x_i|<1$ and can be extended to all $ \mbox{Re}\{ x_i\}<1$ by analytic continuation. Furthermore, it is easy to show that ${\phantom{j}}_1 \mathcal{F}^{(\alpha)}_{0}(a;x;y)$  converges   for   all $ | x_i y_j|<1$. We will prove that the series ${\phantom{j}}_1 \mathcal{F}^{(\alpha)}_{0}(a;x;y)$  can be extended to all $ \mbox{Re}\{ x_i y_j\}<1$ by analytic continuation in the subsequent subsection.

When $\alpha=1$ there exist closed-form expressions  for the hypergeometric series \eqref{hfpq2} and \eqref{hfpq}.

\begin{proposition}[Closed form for $\alpha=1$] \label{alpha=1}
\begin{align}\label{hfpq2-1}
    {\phantom{j}}_p\mathcal{F}^{(1)}_{q}&(a_1+n-1,\ldots,a_p+n-1;b_1+n-1,\ldots,b_q+n-1;x;y)\nonumber\\
    &=  \prod_{i=0}^{n-1}\frac{i!\,(b_1)_{i} \cdots (b_q)_{i} }{(a_1)_{i} \cdots (a_p)_{i}} \frac{\det[_{p}f_{q}(x_{j}y_{k})]}{\det[x_j^{k-1}]\,\det[y_j^{k-1}]},
\end{align}
and
\begin{align}\label{hfpq-1}
    {\phantom{j}}_pF^{(1)}_{q}&(a_1+n-1,\ldots,a_p+n-1;b_1+n-1,\ldots,b_q+n-1;x)\nonumber\\
    &=                            \prod_{i=0}^{n-1}\frac{i!\,(b_1)_{i} \cdots (b_q)_{i} }{(a_1)_{i} \cdots (a_p)_{i}}
    \frac{\det[{x^{k-1}_{j}} { {_{p}f_{q}^{(k-1)}}}(x_{j})]}{\det[x_j^{k-1}]}.
\end{align}
Here
$$_{p}f_{q}(z)=\sum_{k=0}^{\infty}\frac{(a_1)_{k} \cdots (a_p)_k}{(b_1)_{k} \cdots (b_q)_k } \frac{z^{k}}{k!}.$$

\end{proposition}
\begin{proof} First note that for $\alpha=1$ the Jack polynomial $P_{\kappa}^{(1)}(x)$ is equal to  the Schur polynomial $s_{\kappa}^{(1)}(x)$ which is defined by
$$s_{\kappa}^{(1)}(x)=\frac{\det[ x_{j}^{\kappa_{n+1-k}+k-1}]}{\det[x_j^{k-1}]}.$$
Let $$c_{k}=\frac{(a_1)_{k} \cdots (a_p)_k}{(b_1)_{k} \cdots (b_q)_k } \frac{1}{k!},$$
 by Theorem 1.2.2 \cite{hua}, we have the expansion
 \begin{align}
 \frac{\det[_{p}f_{q}(x_{j}y_{k})]}{\det[x_j^{k-1}]\,\det[y_j^{k-1}]}=\sum_{\kappa}  \prod_{i=1}^{n}c_{\kappa_i+n-i} \, P_{\kappa}^{(1)}(x)P_{\kappa}^{(1)}(y).
 \end{align}

 Since $ \prod_{i=1}^{n}c_{\kappa_i+n-i}$ equals to
$$
 \prod_{i=1}^{n}\frac{ (a_1)_{n-i} \cdots (a_p)_{n-i} }{(b_1)_{n-i} \cdots (b_q)_{n-i}}\frac{1}{(1)_{n-i}}
  \frac{ [a_1+n-1]^{(1)}_{\kappa} \cdots [a_p+n-1]^{(1)}_{\kappa} }{[b_1+n-1]^{(1)}_{\kappa} \cdots [b_q+n-1]^{(1)}_{\kappa}}\frac{1}{[n]^{(1)}_{\kappa}}
$$
 and $$h_{\kappa}^{(1)} P_{\kappa}^{(1)}(1^{(n)})=[n]^{(1)}_{\kappa},$$
 the equation  \eqref{hfpq2-1} follows.

Let all $y_j$ go to 1 and take the derivative  we then obtain   \eqref{hfpq-1} from \eqref{hfpq2-1}, or by Theorem 1.2.4 \cite{hua}.
\end{proof}

\subsection{Invariance property with  application}\label{subsecinvariance}

To prove scaling limits for classical $\beta$-ensembles the translation  invariance property of two special  hypergeometric functions  has been proven to be of vital importance (see \cite[Proposition 2.2]{dl}) and will also play a key role in the present article. Explicitly, for  complex numbers $a, b$    and the variables $x=(x_{1},  \ldots, x_{n})$, let
        \be (a^{(n)})=(\overbrace{a,\ldots,a}^{n}),\qquad \frac{b+a x}{1+cx}=(\frac{b+a x_{1}}{1+c x_1},  \ldots, \frac{b+a x_{n}}{1+c x_n}),\nonumber\ee
  then the following translation  invariance holds
\be \label{vip0}{\phantom{j}}_1\mathcal{F}^{(\alpha)}_0(a;b+x;y)=\prod_{j=1}^{n}(1-by_j)^{-a} \,{\phantom{j}}_1\mathcal{F}^{(\alpha)}_0(a;x;\frac{y}{1-by}), \ee
  \be \label{vip1}{\phantom{j}}_0\mathcal{F}^{(\alpha)}_0(a+x;b+y)=\exp\{nab+a p_{1}(y)+b p_{1}(x)\}\,{\phantom{j}}_0\mathcal{F}^{(\alpha)}_0(x;y) \ee
and further
  \begin{align}
&{\phantom{j}}_0 \mathcal{F}_{0}^{(\alpha)}(x_1,\ldots,x_n;a^{(k)}, b^{(n-k)})\nonumber\\
&=e^{b p_1(x)} {\phantom{j}}_1 F^{(\alpha)}_{1}(k/\alpha; n/\alpha; (a-b)x_1,\ldots,(a-b)x_n)\label{vip2}\\ \displaystyle
&=e^{(a-b)kx_1+b p_1(x)} {\phantom{j}}_1 F^{(\alpha)}_{1}(k/\alpha; n/\alpha; (a-b)(x_2-x_1),\ldots,(a-b)(x_n-x_1)).\label{vip3}
\end{align}
See Proposition 2.2 and Corollary 2.3 of \cite{dl} for the proofs of \eqref{vip0}--\eqref{vip3}.

We now give   three application examples of the translation formulas \eqref{vip0} and \eqref{vip1}. The first one is concerned with analytic continuation of the series ${\phantom{j}}_1 \mathcal{F}^{(\alpha)}_{0}$.

   \begin{proposition}\label{analytic10} The domain of the series ${\phantom{j}}_1 \mathcal{F}^{(\alpha)}_{0}(a;x;y)$  can be extended by analytic continuation to some region at least containing $x, y\in\mathbb{C}^n$ such that  \be \label{largedomain} \mathrm{Re}\{ x_i y_j\}<1,\  y_j\geq 0,  \qquad \forall\, i,j=1, \ldots, n.\ee
     \end{proposition}

  \begin{proof} We just deal with the points $x, y\in\mathbb{C}^n$ satisfying   \eqref{largedomain}.

 First, for  $ | x_i y_j|<1$ ($i,j=1, \ldots, n$), the series converges absolutely,  as  discussed   in Subsection \ref{subsecseries}. Otherwise,   there exist
  $x_i, y_j$ such that $\mathrm{Re}\{ x_i y_j\}<1,\  y_j\geq 0$ but $ | x_i y_j|\geq 1$. Write
   $$b_0=\max\Big\{\frac{| x_i y_j|^{2}-1}{2y_j(1-\mathrm{Re}\{ x_i y_j\})}:1\leq i, j\leq n\Big\},$$
   then $b_0\geq 0$.

  Take $b>b_0$, it's easy to verify  that
\be \label{<1}\big|(b+x_i)\frac{y_j}{1+by_j}\big|<1,\qquad \forall\, i,j=1, \ldots, n.\ee
On the other hand, by \eqref{vip0}
 \be \label{explicitcontinuation}{\phantom{j}}_1\mathcal{F}^{(\alpha)}_0(a;x;y)=\prod_{j=1}^{n}(1+by_j)^{-a} \,{\phantom{j}}_1\mathcal{F}^{(\alpha)}_0(a;b+x;\frac{y}{1+by}), \ee
 therefore by \eqref{<1}  the RHS of \eqref{explicitcontinuation} gives an analytic continuation.
%
%
  \end{proof}

 The second one  concerns    the positivity of ${\phantom{j}}_1\mathcal{F}^{(\alpha)}_0$ and ${\phantom{j}}_0\mathcal{F}^{(\alpha)}_0$, which appear as
 one factor of the joint eigenvalue densities for spiked Jacobi/Wishart  $\beta$-ensembles and Gaussian  $\beta$-ensembles with
 source  \cite{des,de,dek,forspiked,wang}.   We remark that the positivity of ${\phantom{j}}_0\mathcal{F}^{(\alpha)}_0$ has been proved
 by R\"{o}sler (see Corollary 4.9 \cite{rosler}), but the proof here is quite different.
    \begin{proposition}\label{positivity}  For all $ x_{i},  y_{j} \in \mathbb{R}$ ($i,j=1, \ldots, n$),
\be \label{00positivity}{\phantom{j}}_0\mathcal{F}^{(\alpha)}_0(x;y)>0,\ee
and for all $ x_{i} y_{j}<1, y_j\geq 0 $ ($i,j=1, \ldots, n$) and $a\geq \frac{n-1}{\alpha}$,
\be \label{10positivity} {\phantom{j}}_1\mathcal{F}^{(\alpha)}_0(a;x;y)>0.\ee
\end{proposition}

Note that by Proposition \ref{analytic10} the function in \eqref{10positivity}  makes sense in  the assumed region of the argument.
\begin{proof} For \eqref{00positivity}, it holds when all $ x_{i},  y_{j} \geq 0$  ($i,j=1, \ldots, n$) by definition, since the coefficients of Jack polynomials are
nonnegative (see Theorem 1.1 \cite{ks97}). Otherwise, if some  $ x_{i}<0$ or $y_{j}<0$ , set $$a=\max\{-x_1, \ldots, -x_n, 0\}, \qquad b=\max\{-y_1, \ldots, -y_n, 0\},$$
then $a+x_k\geq 0$ and $b+y_l\geq 0$ for all $k, l$. By \eqref{vip1} we have
$${\phantom{j}}_0\mathcal{F}^{(\alpha)}_0(x;y) =\exp\{-nab-a p_{1}(y)-b p_{1}(x)\}\, {\phantom{j}}_0\mathcal{F}^{(\alpha)}_0(a+x;b+y)>0.$$

For \eqref{10positivity}, it holds when all $ x_{i}y_{j}<1, x_{i}\geq 0,  y_{j} \geq 0$  ($i,j=1, \ldots, n$), since all the coefficients are non-negative by definition. Otherwise, there exists  some  $ x_{i}<0$, set  $b=\max\{-x_1, \ldots, -x_n\}$, then $b>0$ and $b+x_j\geq 0$ for all $j$.  Since $(b+x_i)\frac{y_j}{1+by_j}<1$ by the assumption, application of \eqref{vip0} gives
$$ {\phantom{j}}_1\mathcal{F}^{(\alpha)}_0(a;x;y)=\prod_{j=1}^{n}(1+by_j)^{-a} \,{\phantom{j}}_1\mathcal{F}^{(\alpha)}_0(a;b+x;\frac{y}{1+by})>0. $$
 \end{proof}
The last one is concerned  with Dubbs-Edelman $\beta$-MANOVA models with general covariance \cite{de}, which may  also be called spiked $\beta$-Jacobi ensembles. The  $\beta$-MANOVA model builds on the $\beta$-Wishart ensemble which is introduced in \cite{dek,forspiked,wang}, see \cite{de} for the detailed definition. However, for $\beta=1,2,4$, it can be defined as follows, see \cite{de} for the three cases and  \cite{james} for $\beta=1$ case.  Let $X$ be an $m\times n$ matrix and $Y$ be a $p\times n$ matrix with independent real, complex, quaternion normal variables,   and let $\Sigma=\mathrm{diag}\{\sigma_1, \ldots, \sigma_n\}$ with $\sigma_j>0$,  the random matrix $\Lambda=Y^{*}Y(Y^{*}Y+\sqrt{\Sigma}X^{*}X\sqrt{\Sigma})^{-1}$ refers to the $\beta$-MANOVA model.

Let $\lambda_1, \ldots, \lambda_n$ be generalized  eigenvalues of  $\beta$-MANOVA models (that is, the  generalized  eigenvalues of $\Lambda$ when $\beta=1,2,4$),  Dubbs and Edelman  have derived the joint density function for generalized  eigenvalues \cite[Theorem 1.1]{de}. We here just give another expression of the density function after application of the equation \eqref{vip0} and Proposition \ref{analytic10}. Our expression is more similar in form to that of the Jacobi $\beta$-ensemble.

\begin{proposition}\label{pdfmanova} Let $\lambda_1, \ldots, \lambda_n$ be generalized  eigenvalues of  $\beta$-MANOVA models with covariance $\Sigma=\mathrm{diag}\{\sigma_1, \ldots, \sigma_n\}$ defined in \cite{de}, and let $p, m>n-1$. Then the  joint density function for eigenvalues   is equal to
\begin{align}  P_{n\beta}(\lambda)=\frac{\mathcal{K}_{p+m,n}^{(\beta)}}{\mathcal{K}_{p,n}^{(\beta)}\,\mathcal{K}_{m,n}^{(\beta)}}\prod_{i=1}^{n}\lambda_{i}^{\frac{p-n+1}{2}\beta-1}(1-\lambda_{i})^{\frac{m-n+1}{2}\beta-1} \ |\Delta_{n}(\lambda)|^{\beta}\nonumber\\
\times  {\phantom{j}}_1\mathcal{F}^{(2/\beta)}_0\big(\frac{p+m}{2}\beta;\lambda;1-\sigma\big), \qquad 0<\lambda_1, \ldots, \lambda_n<1,\label{10factor}\end{align}
 where
 $$ \mathcal{K}_{m,n}^{(\beta)}= 2^{\beta mn/2} (\Gamma(\beta/2))^{-n} \prod_{i=1}^{n}\Gamma(\beta i/2) \Gamma(\beta (m-n+i)/2).$$

\end{proposition}

Note that    the factor ${\phantom{j}}_1\mathcal{F}^{(2/\beta)}_0$ of the joint density in \eqref{10factor}  makes sense by Proposition \ref{analytic10} and is positive by Proposition \ref{positivity}. Also note that the notation ${\phantom{j}}_1F^{(\beta)}_0$ is used in \cite{de} rather than ${\phantom{j}}_1\mathcal{F}^{(2/\beta)}_0$ here, but they refer to the same function.

\begin{proof} Theorem 1.1 \cite{de} shows the joint density function equals
\begin{align}  P_{n\beta}(\lambda)=\frac{\mathcal{K}_{p+m,n}^{(\beta)}}{\mathcal{K}_{p,n}^{(\beta)}\,\mathcal{K}_{m,n}^{(\beta)}}\prod_{i=1}^{n}\lambda_{i}^{\frac{p-n+1}{2}\beta-1}(1-\lambda_{i})^{-\frac{p+n-1}{2}\beta-1} \ |\Delta_{n}(\lambda)|^{\beta}\nonumber\\
\times  {\phantom{j}}_1\mathcal{F}^{(2/\beta)}_0\big(\frac{p+m}{2}\beta;-\frac{\lambda}{1-\lambda};\sigma\big). \label{depdf}\end{align}
By  \eqref{vip0}  we have
\begin{align}  {\phantom{j}}_1\mathcal{F}^{(2/\beta)}_0\big(\frac{p+m}{2}\beta;-\frac{\lambda}{1-\lambda};\sigma\big)={\phantom{j}}_1\mathcal{F}^{(2/\beta)}_0\big(\frac{p+m}{2}\beta;\frac{\lambda}{1-\lambda};-\sigma\big)\\
=\prod_{i=1}^{n} (1-\lambda_{i})^{\frac{p+m}{2}\beta}{\phantom{j}}_1\mathcal{F}^{(2/\beta)}_0\big(\frac{p+m}{2}\beta; \lambda;1-\sigma\big), \end{align}
with \eqref{depdf} which completes the proof.
\end{proof}

\section{Spectrum singularity}
\label{scaling}
In this section we consider the case of the fixed constant $b$. The starting point is the following Selberg correlation integrals (see, e.g., \cite[Proposition 13.1.2]{forrester})
\begin{align}
 \int_{[0,2\pi)^{N}}|\Delta_{N}(e^{i\theta})|^{2/\alpha}\prod_{j=1}^{N} e^{i\frac{a'-b'}{2} (\theta_{j}-\pi) }|1-e^{i\theta_{j}}|^{a'+b'} \prod_{j=1}^m \phi(s_{j})\,d^{N}\theta\nonumber\\
 =M_{N}( a', b';1/\alpha)\!{\phantom{j}}_2F_1^{(1/\alpha)}(-N,\alpha b';-N+1-\alpha(1+a'); s_{1} ,\ldots,s_{m}).
 \label{sci}
\end{align}

We first give an expression of the expected products of characteristic polynomials in terms of the generalized hypergeometric functions, which will also be used later in the subsequent sections.     \begin{proposition}\label{hgfexpression}
\begin{align}
&K_{b,N}(s;t)
 =C_{N} \,\prod_{k=1}^n (\overline{t_{k}})^{N} \times\nonumber\\
  &\!{\phantom{j}}_2F_1^{(\beta/2)}(-N,(2/\beta)(b+n);(2/\beta)(\bar{b}+b+m+n); 1-s_{1} ,\ldots,1-s_{m+n}),
\end{align}
where $K_{b,N}(s;t)$ is given in \eqref{productmean}, $s_{m+j}= (\overline{t_{j}})^{-1}$ for $j=1,\ldots,n$ and
\begin{align} C_{N}&=\prod_{j=0}^{m-1}
\frac{\Gamma((2/\beta)(1+ \bar{b}+j))\, \Gamma(N+(2/\beta)(1+ \bar{b}+b+n+j))}
{\Gamma((2/\beta)(1+ \bar{b}+b+n+j))\, \Gamma(N+(2/\beta)(1+ \bar{b}+j))}\nonumber\\
&\times \prod_{j=0}^{n-1}
\frac{\Gamma((2/\beta)(1+ b+j))\, \Gamma(N+(2/\beta)(1+ \bar{b}+b+j))}
{\Gamma((2/\beta)(1+ \bar{b}+b+j))\, \Gamma(N+(2/\beta)(1+ b+j))}.\label{constantC}\end{align}
\end{proposition}
\begin{proof}
Since  $$\overline{\phi(z)}=(\bar{z})^{N}\phi(1/\bar{z} )\prod_{j=1}^{N}e^{-i(\theta_{j}-\pi)},$$
set $s_{m+j}= (\overline{t_{j}})^{-1}$, by \eqref{sci} with $a'=\bar{b}-n, b'=b+n$ and $2/\alpha=\beta$,  we get
\begin{align*}&K_{b,N}(s;t) =\prod_{k=1}^n (\overline{t_{k}})^{N} \E[\, \prod_{j=1}^{m+n} \phi(s_{j}) \prod_{k=1}^N e^{-in(\theta_{k}-\pi)}\,]=\prod_{k=1}^n (\overline{t_{k}})^{N} \times\nonumber\\
& \frac{M_{N}( \bar{b}-n, b+n;\beta/2)}{M_{N}( \bar{b}, b;\beta/2)}\!{\phantom{j}}_2F_1^{(\beta/2)}(-N,(2/\beta)(b+n);-N+1-(2/\beta)(\bar{b}-n+1); s).\end{align*}

Notice
$$\frac{M_{N}( \bar{b}-n, b+n;\beta/2)}{M_{N}( \bar{b}, b;\beta/2)}=\prod_{j=0}^{n-1}
\frac{\Gamma((2/\beta)(1+ b+j))\, \Gamma(N+(2/\beta)(\bar{b}-j))}
{\Gamma(N+(2/\beta)(1+ b+j))\, \Gamma((2/\beta)(\bar{b}-j))},$$
by Proposition 13.1.7  and the formula (13.14) in \cite[Chap. 13]{forrester},  simple manipulation of the gamma functions gives  the desired result.
\end{proof}

Next, we define a    multivariate   function which appears as a scaling limit near the spectrum singularity.
 \begin{definition} \label{defsingularty} For $\alpha>0$ and $\textrm{Re}\{b\}>-1/2$,  we define a  multivariate   function
\begin{align}\label{ssdef}
  S_{b}^{(\alpha)}(x_1,\ldots,x_{m+n})&=\gamma_{m,n}(b,2/\alpha) \prod_{k=1}^{m+n} e^{-i x_{k}/2 }\, \times\nonumber\\
    &\!{\phantom{j}}_1F_1^{(\alpha )}((b+n)/\alpha;(\bar{b}+b+m+n)/\alpha; ix_{1},\ldots,ix_{m+n})\end{align}
where the constant  \begin{align} \gamma_{m,n}(b,2/\alpha)&=\prod_{j=0}^{m-1}
\frac{\Gamma((1+ \bar{b}+j)/\alpha) }
{\Gamma((1+ \bar{b}+b+n+j)/\alpha) } \prod_{j=0}^{n-1}
\frac{\Gamma((1+ b+j)/\alpha) }
{\Gamma((1+ \bar{b}+b+j)/\alpha)}.\label{universalconstantcp}\end{align}.
\end{definition}

\begin{remark} Assume that  $m=n$ and all $x_j$ are real variables, then $S_{b}^{(\alpha)}(x)$ is a real-valued function, that is
$$S_{b}^{(\alpha)}(x)=\overline{S_{b}^{(\alpha)}(x)}.$$ This follows from
$\gamma_{m,m}(b,2/\alpha) =\gamma_{m,m}(\bar{b},2/\alpha)$ and the Kummer relation (see Proposition 3.2 \cite{yan})
$$\prod_{k=1}^{n} e^{ x_{k}}{\phantom{j}}_1F_1^{(\alpha )}(c-a;c; - x_{1},\ldots,-x_{n})={\phantom{j}}_1F_1^{(\alpha )}(a;c; x_{1},\ldots,x_{n}).$$
\end{remark}
In particular for $\alpha=1$, application of \eqref{hfpq-1} in Proposition \ref{alpha=1} gives    a determinant expression of the function $S_{b}^{(1)}(x)$, that is
\begin{align}
  S_{b}^{(1)}(x_1,\ldots,x_{m+n}) &=\gamma_{m,n}(b,2) \prod_{j=0}^{m+n-1}\frac{j!\,(\bar{b}+b+1)_{j} }{(b-m+1)_{j}}  \nonumber \\
  & \times e^{-i \sum_{k=1}^{m+n}x_{k}/2 }\, \frac{\det[{(ix_{j})^{k-1}} {{\phantom{j}}_{1}f_{1}^{(k-1)}}(ix_{j})]}{\det[(ix_j)^{k-1}]},
\end{align}
where
\be _{1}f_{1}(z):= {_{1}f_{1}}(b-m+1;\bar{b}+b+1;z)=\sum_{k=0}^{\infty}\frac{(b-m+1)_{k}}{(\bar{b}+b+1)_{k} } \frac{z^{k}}{k!}.\ee
  Further, if $m=n=1$, then
   \begin{align}
 & S_{b}^{(1)}(x_1,x_{2})=\frac{\Gamma(\bar{b}+1)\Gamma(b+1)}{\Gamma(\bar{b}+b+1)\Gamma(\bar{b}+b+2)} \,e^{-i  (x_{1}+x_2)/2 } \frac{1}{x_2 -x_1}\, \nonumber \\
  & \times \left(x_{2}\, {_{1}f_{1}}(b;\bar{b}+b+1;ix_1)\,{_{1}f_{1}}(b+1;\bar{b}+b+2;ix_2)-(x_{1}\leftrightarrow x_2)\right),
 \end{align}
 which is a new kernel and reduces to the sine kernel 
 when $b=0$.

    Another special case is that  when $m=n=1$ and $b=0$ for general $\alpha$,  $S_{0}^{(\alpha)}(x_1,x_{2})$ can be expressed by the Bessel function
    $(x_2-x_1)^{-\frac{1}{\alpha}+\frac{1}{2}}J_{\frac{1}{\alpha}-\frac{1}{2}}(x_2-x_1)$, see Proposition 5.2 \cite{dl}.

We now state the main results of this section   as follows.
 \begin{theorem}[Limit at zero]\label{limitatzero} Let   $\rho$ be a nonzero constant. Introducing the scaled variables $$s_{j}=e^{i\frac{x_{j}}{\rho N}}, \ j=1,\ldots,m \quad \mbox{and} \quad t_{k}=e^{i\frac{x_{m+k}}{\rho N}}, \  k=1,\ldots,n$$
with real variables  $x_{1}, \ldots, x_{m+n}$, we have \begin{multline}
\lim_{N\rightarrow \infty} N^{-2(mn+bm+\bar{b}n)/\beta}K_{b,N}(s;t)=\\
\prod_{j=1}^m e^{i x_{j}/(2\rho) } \prod_{k=1}^n e^{-i x_{m+k}/(2\rho) }S_{b}^{(\beta/2)}(x_{1}/\rho,\ldots,ix_{m+n}/\rho)
 \label{limitfuctionzero}
\end{multline}
  where $K_{b,N}(s;t)$ is given in \eqref{productmean}. 
\end{theorem}

Note that when $b=0$ ${\phantom{j}}_0 \mathcal{F}_{0}^{(\beta/2)}(x_1,\ldots,x_{m+n};1^{(m)}, (-1)^{(n)})$ appears as a scaling limit in
the non-trivial part $S_{b}^{(\beta/2)}(x)$    on the RHS of \eqref{limitfuctionzero} for any integers $m$ and $n$,  after  application of \eqref{vip2}.
\begin{proof}
Notice the asymptotic formula of the gamma functions $$\frac{\Gamma(N+b)}{\Gamma(N+a)}\sim N^{b-a} \ \mathrm{as}\  N\rightarrow \infty,$$
 together with the following limit property (cf. (13.5), \cite[Chap. 13]{forrester})
 $$\lim_{b'\rightarrow \infty}{\phantom{j}}_2F_1^{(\alpha)}(a', b';c'; s_{1}/b',\ldots,s_{m}/b')={\phantom{j}}_1F_1^{(\alpha)}(a';c'; s_{1},\ldots,s_{m}),$$ application of Proposition \ref{hgfexpression} gives the desired result.
\end{proof}

When $\beta$ is an even integer, the scaling limit of correlation functions immediately follows from Theorem \ref{limitatzero}.
 \begin{coro}[Correlations at zero for even $\beta$]\label{cflimitatzero} Let   $\rho$ be a nonzero constant. Introducing the scaled variables $$r_{j}= \frac{y_{j}}{\rho N} ,\quad    j=1,\ldots,k$$
with real variables  $y_{1}, \ldots, y_{k}$ in \eqref{defpcf}, we have \be
\lim_{N\rightarrow \infty} \big(\frac{1}{\rho N}\big)^{k} R^{(k)}_{b,N}\big(\frac{y_{1}}{\rho N},\ldots,\frac{y_{k}}{\rho N}\big)
=R^{(k,\beta/2)}_{b,\rho}(y_1,\ldots,y_k),\nonumber\ee
where
\begin{multline}
R^{(k,\beta/2)}_{b,\rho}(y_1,\ldots,y_k)=c_{k}(b,\beta) \,(2\pi\rho)^{-k}e^{i \frac{b-\bar{b}}{2}k\pi }\prod_{j=1}^k\big(| y_j/\rho|^{\bar{b}+b} e^{i \frac{\beta}{2\rho}y_{j}}\big) \ |\Delta_{k}(y/\rho)|^{\beta} \, \times\\
\!{\phantom{j}}_1F_1^{(\beta/2)}((2/\beta)b+k;(2/\beta)(\bar{b}+b)+2k; -iy_{1}/\rho,\cdots,-iy_{1}/\rho, \ldots, -iy_{k}/\rho).\label{bcorrelation}
\end{multline}
Here in the argument of ${\phantom{j}}_1F_1^{(\beta/2)}$ each $-iy_{j}/\rho$ ($j=1,\ldots,k$) occurs $\beta$ times, and \begin{multline} c_{k}(b,\beta)=(\beta/2)^{(\bar{b}+b)k+\beta k(k-1)/2}\big(\Gamma(1+\beta/2)\big)^{k} \times\\
\prod_{j=0}^{2k-1}\frac{1}{\Gamma(1+ \bar{b}+b+\beta j/2)}\  \prod_{j=0}^{k-1}\Gamma( 1+ \bar{b}+\beta j/2) \Gamma( 1+ b+\beta j/2).\label{universalconstantcfbulk}\end{multline}
\end{coro}
\begin{proof}Let $m=\beta k/2$. Introduce $s_j=e^{i \frac{x_j}{\rho N}}, \ j=1,\ldots,m$ and
$$x_{l+ \beta(j-1)/2}=-y_j, \quad j=1,\ldots,k, l=1,\ldots, \beta/2,$$
by  use of the formula \eqref{variablescptopcf} simple manipulation gives  as $N\rightarrow \infty$
\begin{multline}
 \big(\frac{1}{\rho N}\big)^{k} R^{(k)}_{b,N}\big(\frac{y_{1}}{\rho N},\ldots,\frac{y_{k}}{\rho N}\big)\sim \frac{ M_{N}(\bar{b},b;\beta/2)}{M_{k+N}(\bar{b},b;\beta/2)} \rho^{-k}\times\\
  N^{-(\bar{b}+b)k-\beta k(k-1)/2} \, e^{i \frac{b-\bar{b}}{2}k\pi }\prod_{j=1}^k | y_j/\rho|^{\bar{b}+b}  \ |\Delta_{k}(y/\rho)|^{\beta} \,
K_{b,N}(s;s).\label{cf1step}
\end{multline}

Notice the notation $M_N$ specified by \eqref{Mconstant} and one easily derives $$\frac{ M_{N}(\bar{b},b;\beta/2)}{M_{k+N}(\bar{b},b;\beta/2)}\sim (2\pi)^{-k}\big(\Gamma(1+\beta/2)\big)^{k} (\beta N/2)^{-\beta k/2}.$$
On the other hand, for Gauss's multiplication formulas of the gamma functions
\be \label{gaussmul}\prod_{j=0}^{l-1} \Gamma (a+  \frac{j}{l} )=l^{-la+\frac{1}{2}}(2\pi)^{\frac{l-1}{2}} \Gamma(la), \qquad l\in \mathbb{N}, \ee
let $l=\beta/2$ and  we get (cf. \eqref{universalconstantcp})
\begin{align} \gamma_{m,m}(b, 4/\beta)&=
\frac{ \prod_{p=0}^{k-1}\prod_{j=0}^{-1+\beta/2}\Gamma( \frac{1+ \bar{b}}{\beta/2}+p+\frac{j}{\beta/2} )\,\Gamma( \frac{1+ b}{\beta/2}+p+\frac{j}{\beta/2} ) }
{ \prod_{p=0}^{2k-1}\prod_{j=0}^{\frac{\beta}{2}-1}\Gamma( \frac{1+ \bar{b}+b}{\beta/2}+p+\frac{j}{\beta/2} ) } \nonumber\\
&=
\frac{\prod_{p=0}^{k-1} \Gamma(1+ \bar{b} + \beta p/2)\, \Gamma(1+ b + \beta p/2) }
{\prod_{p=0}^{2k-1} \Gamma(1+ \bar{b} +b+ \beta p/2) }.  \end{align}

Together with \eqref{cf1step}, the desired result immediately follows from Theorem \ref{limitatzero}.
\end{proof}

We remark that in the  circular $\beta$-ensemble, i.e., $b=0$, Corollary \ref{cflimitatzero} is actually identical to the bulk of the Gaussian, Laguerre and Jacobi ensembles \cite{dl}, and has been proved
by Forrester \cite{forrester92} (see also Proposition 13.2.3 of \cite{forrester}). For  $\beta=1, 2, 4$ and fixed real $b$, the scaling limits for
point correlation functions near the spectrum singularity (at zero) were proved by Forrester and Nagao \cite{fn01}, see also
Chapter 7.2.6, \cite{forrester} for the limiting kernel and relationship with the sine kernel in the $\beta=2$ case. For  $\beta=2$ and fixed real $b$, the
same universal correlation functions as in  Corollary \ref{cflimitatzero} appear in a  wide  class of unitary ensembles with singularity at the origin, see
\cite{admn97,admn98,kv,ns} and references therein.

\begin{remark}  Theorem \ref{limitatzero} and Corollary \ref{cflimitatzero} just involve the local limit at the angle $\theta=0$   of the unit circle.
However, when $b=0$, the local limit still holds true near every angle of the unit circle because of the relation  $K_{0,N}(e^{i\theta}s;e^{i\theta}t)=K_{0,N}(s; t)$ for any real $\theta$. But for the general fixed $b$, our method may not  work at the other angle than zero.

\end{remark}

\section{Limits for $b=\beta Nd/2$}
We  devote ourselves to the case of $b=\beta Nd/2$  
 and prove the local limits in the bulk and at the soft edge. In this section let $\alpha=\beta/2$,  for convenience, $\alpha$ and $\beta$   will be used in alternation.
 \subsection{General procedure}\label{general}
   Using Proposition \ref{hgfexpression} with the same notation $s_{m+j}= (\overline{t_{j}})^{-1}$ ($j=1,\ldots,n$), if $\textrm{Re}\{b\}>\max\{m,n\}-1$, we then have the following integral  representation due to Yan \cite{yan} \begin{align}
K_{b,N}&(s;t)
 =C_{N} \,\prod_{k=1}^n (\overline{t_{k}})^{N} \frac{1}{S_{m+n}(\nu_1,\nu_2;1/\alpha)}\, \times\nonumber\\
&
\int_{[0,1]^{m+n}} {\!\!\!{\phantom{j}}_1\mathcal{F}_{0}^{(\alpha)}(-N;1-s;y)}\prod_{j=1}^{m+n}y_{j}^{\nu_1}(1-y_{j})^{\nu_2}\,|\Delta_{m+n}(y)|^{2/\alpha}\,d^{m+n}y,
\end{align}
where $\nu_{1}= -1+(b-m+1)/\alpha, \, \nu_{2}= -1+(\bar{b}-n+1)/\alpha$ and
 \be S_{m+n}(\nu_1,\nu_2;1/\alpha)=\prod_{j=0}^{m+n-1}
\frac{\Gamma(1+ (1+j)/\alpha) \Gamma(1+\nu_{1}+j/\alpha) \Gamma(1+\nu_{2}+j/\alpha)}
{\Gamma(1+1/\alpha)\Gamma(2+\nu_{1}+\nu_{2}+(m+n+j-1)/\alpha)}.\label{selbergc}\ee

Introduce the scaled variables $$s_{j}=e^{i\theta+i\frac{x_{j}}{\rho N}}, \ j=1,\ldots,m+n,$$
where $\theta$ is  a spectral parameter  that  allows  us to select the part of  the spectrum we are going to study, and $\rho$ is to be determined. For convenience, let the new spectral parameter \be u= \frac{1}{1-e^{i\theta}},\ee then $1-s_j=\frac{1}{u}+w_{j}$ where $w_j=e^{i\theta}(1-e^{i\frac{x_{j}}{\rho N}})$. We also assume $b=\alpha N d$ where $\textrm{Re}\{d\}>0$. Application of \eqref{vip0} gives
\begin{align}
K_{b,N}&(s;t)
 =\frac{C_{N} e^{-inN\theta}}{S_{m+n}(\nu_1,\nu_2;1/\alpha)} \,\prod_{k=1}^n e^{-ix_{m+k}/\rho} \, \times\nonumber\\
&
\int_{[0,1]^{m+n}} \exp\{-N\sum_{j=1}^{m+n}p(y_j)\} \,|\Delta_{m+n}(y)|^{2/\alpha} \tilde{Q}(y)\,d^{m+n}y,
\end{align}
where
\be p(y_j)=-d \log y_j\ -\bar{d}\log(1-y_j)-\log(1-u^{-1}y_j),\ee
and
\be \tilde{Q}(y)=\prod_{j=1}^{m+n}y_{j}^{ \frac{1-m}{\alpha}-1}(1-y_{j})^{ \frac{1-n}{\alpha}-1 }{\!\!\!{\phantom{j}}_1\mathcal{F}_{0}^{(\alpha)}(-N;w;\frac{y}{1-u^{-1}y})}.\ee

Notice as $N\rightarrow \infty$ one has $w_j\sim -ie^{i\theta} x_{j}/(\rho N)$, and
\be \tilde{Q}(y)=\prod_{j=1}^{m+n}y_{j}^{ \frac{1-m}{\alpha}-1}(1-y_{j})^{ \frac{1-n}{\alpha}-1 }  \Big({\!\!\!{\phantom{j}}_0\mathcal{F}_{0}^{(\alpha)}(x/\rho;\frac{ie^{i\theta} y}{1-u^{-1}y})}+O(\frac{1}{N})\Big), \nonumber\ee
  which shows
  \begin{align}\label{eq1}
K_{b,N}&(s;t)
\sim \frac{C_{N} e^{-inN\theta}}{S_{m+n}(\nu_1,\nu_2;1/\alpha)} \,\prod_{k=1}^n e^{-ix_{m+k}/\rho} \, I_{N}.
\end{align}
Here
\be \label{eq2} I_{N}=\int_{[0,1]^{m+n}} \exp\{-N\sum_{j=1}^{m+n}p(y_j)\} \,|\Delta_{m+n}(y)|^{2/\alpha} Q(y)\,d^{m+n}y
 \ee
and
  \be \label{eq3} Q(y)=\prod_{j=1}^{m+n}y_{j}^{ \frac{1-m}{\alpha}-1}(1-y_{j})^{ \frac{1-n}{\alpha}-1 }  {\!\!{\phantom{j}}_0\mathcal{F}_{0}^{(\alpha)}(x/\rho;\frac{ie^{i\theta}y}{1-u^{-1}y})}.\ee

From now on we suppose that $d$ is a positive real number,   we then have
\be \label{pfunction} p(y_j)=-d \log y_j\ -d \log(1-y_j)-\log(1-u^{-1}y_j).\ee
However, our method may be applicable  to the general case of $\textrm{Re}\{d\}>0$, except that some much more complicated computation must be done.

 We will make use of Corollaries (3.11) and (3.12)  of \cite{dl},  which are based on the steepest descent method for Selberg-type integrals,  to asymptotically evaluate $I_{N}$ as $N\to \infty$.
  In our case, since $$p'(x)=-\frac{d}{x}+\frac{d}{1-x}+\frac{1}{u-x},$$ there are at most two saddle points $x_\pm$, which satisfy \be  \label{twosaddle00} x_\pm=\frac{1}{2}+\frac{1}{2(1+\sin\frac{\theta_\textrm{d}}{2})\sin\frac{\theta}{2}}\Big(\pm\sqrt{\sin^{2}\frac{\theta}{2}-\sin^{2}\frac{\theta_\textrm{d}}{2}}+i\sin\frac{\theta_\textrm{d}}{2}\cos\frac{\theta}{2}\Big) .\ee
 Recall that   $\theta_\textrm{\textrm{d}}\in [0,\pi)$ and is defined by
\be \label{sind}\sin\frac{\theta_\textrm{d}}{2}= \frac{\textrm{d}}{1+\textrm{d}} .\ee
 Also let  $\varphi \in [-\frac{\pi}{2},\frac{\pi}{2}]$ be such that
\be \cos\varphi=\frac{\sqrt{\sin^{2}\frac{\theta}{2}-\sin^{2}\frac{\theta_\textrm{d}}{2}}}{
 \cos\frac{\theta_\textrm{d}}{2} \sin\frac{\theta}{2} }, \qquad \sin\varphi=\frac{ \sin\frac{\theta_\textrm{d}}{2}\cos\frac{\theta}{2}}{
 \cos\frac{\theta_\textrm{d}}{2} \sin\frac{\theta}{2} },\ee
 then we have \be \label{twosaddles}x_+=\frac{1}{2}+\frac{\cos\tfrac{\theta_\textrm{d}}{2}}{2(1+\sin\tfrac{\theta_\textrm{d}}{2})}e^{i\varphi}, \qquad x_-=\frac{1}{2}+\frac{\cos\tfrac{\theta_\textrm{d}}{2}}{2(1+\sin\tfrac{\theta_\textrm{d}}{2})}e^{i(\pi-\varphi)}.\ee
 The nature of the saddle points depends on the value of $\theta$, and we distinguish two cases:
\begin{enumerate}
\item  Two complex saddle points of degree one $x_+, x_-$ given in \eqref{twosaddles} when $\theta\in (\theta_\textrm{d}, 2\pi-\theta_\textrm{d})$, i.e. $\varphi \in (-\frac{\pi}{2},\frac{\pi}{2})$.
\item  One   saddle point of degree two \be \label{one+} x_0=\frac{1}{2}+\frac{i\cos\frac{\theta_\textrm{d}}{2}}{2(1+\sin\frac{\theta_\textrm{d}}{2})}\  \mathrm{for} \ \theta=\theta_\textrm{d}\ee
    while \be \label{one-} x_0=\frac{1}{2}-\frac{i\cos\frac{\theta_\textrm{d}}{2}}{2(1+\sin\frac{\theta_\textrm{d}}{2})}\  \mathrm{for} \ \theta=2\pi-\theta_\textrm{d}.\ee
\end{enumerate}

In the following  subsections, we proceed to compute the leading asymptotic terms for products of characteristic polynomials  by use of  Corollaries 3.11 and 3.12  in \cite{dl}, respectively in the bulk and at  the soft edge of the spectrum.

\subsection{Bulk limits}
 \begin{theorem}[Bulk limit]\label{bulklimits} Let   $$\rho=\frac{1}{2\pi}\frac{\sqrt{\sin^{2}(\theta/2)-\sin^{2}(\theta_\textrm{d}/2)}}{(1-\sin(\theta_\textrm{d}/2))\sin(\theta/2)}, \qquad \theta_\textrm{d}<\theta<2\pi-\theta_\textrm{d},$$   and introduce the scaled variables $$s_{j}=e^{i\theta+i\frac{x_{j}}{\rho N}}, \ j=1,\ldots,m \quad \mbox{and} \quad t_{k}=e^{i\theta+i\frac{x_{m+k}}{\rho N}}, \  k=1,\ldots,n$$
with real variables  $x_{1}, \ldots, x_{m+n}$. Assume that $m+n=2l$  and $b=\beta Nd/2$ with $d>0$, then as $N\rightarrow \infty$ we have \begin{multline}
   \frac{K_{b,N}(s;t)}{\Psi_{N,m,n}}\sim
\exp\Big\{ \frac{1}{2\rho}
\big[
(i-d \cot\tfrac{\theta}{2}  ) \sum_{j=1}^m   x_{j}
 -(i+d \cot\tfrac{\theta}{2}  ) \sum_{j=1}^n   x_{m+j}
\big]
\Big\}  \\ \, \times\gamma_{l}(4/\beta) \  e^{-i\pi \sum_{k=1}^{2l} x_{k}}\,
\!{\phantom{j}}_1F_1^{(\beta/2)}(2l/\beta; 4l/\beta ; 2i\pi x)
\end{multline}
where $K_{b,N}(s;t)$ is defined in \eqref{productmean}, $\gamma_{l}(4/\beta):=\gamma_{l,l}(0,4/\beta)$ is given in \eqref{universalconstantcp} and \begin{align}\Psi_{N,m,n}=(2\pi \rho)^{\tfrac{2l(l+1)}{\beta}-l}e^{i\tfrac{m-n}{2}N \theta+i\tfrac{m-n}{\beta}l (\theta-\pi)}(2\sin\tfrac{\theta}{2})^{-2ldN}N^{\tfrac{2l^{2}}{\beta}}d^{\tfrac{(m-n)^{2}}{2\beta}}\nonumber\\
\times (1+d)^{-2l(1+d)N-\tfrac{m(m+1)+n(n+1)}{\beta}+l}(1+2d)^{l(1+2d)N+\tfrac{l(l+1)}{\beta}-\tfrac{l}{2}}.\end{align}
\end{theorem}
\begin{proof} We now continue the foregoing analysis in Section \ref{general}. According to  Corollary 3.12  in \cite{dl}, we   will  perform some simple calculations as follows.

First, together with \eqref{pfunction} and \eqref{twosaddles}  simple manipulation shows
\begin{align}\label{ptwovalue} p(x_{\pm})=  \mp i(\frac{\pi}{2}+d\varphi)+(1+d)\log\frac{\sqrt{\sin^{2}\frac{\theta}{2}-\sin^{2}\frac{\theta_\textrm{d}}{2}}   \pm i\cos\frac{\theta}{2}}{\cos\frac{\theta_\textrm{d}}{2}}\nonumber\\
-i\frac{\theta}{2}+\log\frac{1+\sin\frac{\theta_\textrm{d}}{2}}{\cos\frac{\theta_\textrm{d}}{2}} +d\log2 +d\log\big(\frac{1+\sin\frac{\theta_\textrm{d}}{2}}{\sin\frac{\theta_\textrm{d}}{2}}\sin\frac{\theta}{2}\big),  \end{align}
and \begin{align} \label{ptwoprime}p&''(x_{\pm})= \frac{4\sqrt{\sin^{2}\frac{\theta}{2}-\sin^{2}\frac{\theta_\textrm{d}}{2}}}{\sin\frac{\theta_\textrm{d}}{2}(1-\sin\frac{\theta_\textrm{d}}{2})^2}\cos\frac{\theta_\textrm{d}}{2}\sin\frac{\theta}{2}\,
(1+\sin\frac{\theta_\textrm{d}}{2})\,\times\nonumber\\
\Big\{&
\pm i\frac{\cos\frac{\theta}{2} (2\sin^{2}\frac{\theta}{2}+\sin\frac{\theta_\textrm{d}}{2}-\sin^{2}\frac{\theta_\textrm{d}}{2})} {\cos\frac{\theta_\textrm{d}}{2} \sin\frac{\theta}{2}(1+\sin\frac{\theta_\textrm{d}}{2} )} +
\frac{2\sin^{2}\frac{\theta}{2}+\sin\frac{\theta_\textrm{d}}{2}-1} {\cos\frac{\theta_\textrm{d}}{2} \sin\frac{\theta}{2}(1+\sin\frac{\theta_\textrm{d}}{2} )}
 \sqrt{\sin^{2}\frac{\theta}{2}-\sin^{2}\frac{\theta_\textrm{d}}{2}}
 \Big\}.  \end{align}

 Secondly, notice  the formulas \eqref{eq2}--\eqref{twosaddle00} and  by  Corollary 3.12 of \cite{dl} we get for even $m+n=2l$
  \begin{align}\label{Iasym1}
I_{N}&\sim  \binom{2l}{l}(\Gamma_{2/\alpha,l})^2\frac{(x_+-x_-)^{2 l^{2}/\alpha}}{(\sqrt{p''(x_+)p''(x_-)})^{l+l(l-1)/\alpha}}\frac{e^{-lN(p(x_+)+p(x_-))}}{N^{l+ l(l-1)/\alpha}}Q(x_{+}^{_{(l)}},x_{-}^{_{(l)}})
\end{align}
where  a special case of Selberg  integrals \cite{forrester}
\be \int_{\mathbb{R}^n}\prod_{k=1}^l e^{- x_k^2/2}\prod_{1\leq i<j \leq l}|x_j-x_i|^{\beta}\,d^{l}x= \Gamma_{\beta,l}
\ee
reads off
\be \label{gaussint}
 \Gamma_{\beta,l}=(2\pi)^{l/2}\prod_{j=1}^l\frac{\Gamma(1+j\beta/2)}{\Gamma(1+\beta/2)}.
\ee
 Since a direct computation gives \be  \frac{ie^{i\theta}x_{\pm}}{1-u^{-1}x_{\pm}}= \frac{1}{2\sin\frac{\theta}{2}}\big(e^{i\frac{\theta}{2}}-\frac{\cos\frac{\theta}{2}}{1-\sin\frac{\theta_\textrm{d}}{2}}\big)\pm
  \frac{i\sqrt{\sin^{2}\frac{\theta}{2}-\sin^{2}\frac{\theta_\textrm{d}}{2} }}{2\sin\frac{\theta}{2}(1-\sin\frac{\theta_\textrm{d}}{2} )},
 \ee
 substitution of   \eqref{twosaddle00}, \eqref{ptwovalue}, \eqref{ptwoprime} into the RHS of \eqref{Iasym1} and application of  the translate property \eqref{vip1},   after minor manipulation   we get
  \begin{align}\label{Iasym2}
I_{N}& \sim  \binom{2l}{l}(\Gamma_{2/\alpha,l})^2 N^{-\frac{l(l-1)}{\alpha}-l}\, \times\nonumber\\
&(2\pi \rho)^{\frac{l(l+1)}{\alpha}-l}  e^{i\frac{l}{2\alpha}(m-n) (\theta-\pi)}  \big( \frac{1-\sin\frac{\theta_\textrm{d}}{2}}{1+\sin\frac{\theta_\textrm{d}}{2}}\big)^{\frac{2l^{2}}{\alpha}} (\tan\frac{\theta_\textrm{d}}{2})^{-\frac{l(l-1)}{\alpha}-l} \, \times\nonumber
\\&\exp\Big\{ -lN
\big[-i\theta+2\log \frac{1+\sin\frac{\theta_\textrm{d}}{2}}{\cos\frac{\theta_\textrm{d}}{2}} +2d\log2+2d\log \frac{(1+\sin\frac{\theta_\textrm{d}}{2})\sin\frac{\theta}{2}}{\sin\frac{\theta_\textrm{d}}{2}}
 \big]
\Big\} \, \times\nonumber\\
&\exp\Big\{ \frac{1}{2\rho}
\big[ (i-d \cot\frac{\theta}{2}  ) \sum_{j=1}^m   x_{j}
 -(i+d \cot\frac{\theta}{2}  ) \sum_{j=1}^n   x_{m+j}
\big]
\Big\}     \, \!{\phantom{j}}_0\mathcal{F}_0^{(\alpha)}(i\pi x;1^{(l)},(-1)^{(l)}). \nonumber
\end{align}

On the other hand, use of  the asymptotic formula for the gamma functions gives as $N\rightarrow \infty$
\begin{eqnarray}
     &&\frac{C_{N}}{S_{m+n}(\nu_1,\nu_2;1/\alpha)}\sim   (2\pi)^{-\frac{m+n}{2}} (1+2d)^{(1+2d)(m+n)N+\frac{(m+n)(m+n+1)}{\beta}-\frac{m+n}{2}}\nonumber\\
   &&\times (1+d)^{-(1+d)(m+n)N-\frac{m(m+1)+n(n+1)}{\beta}+\frac{m+n}{2}} d^{ -d (m+n)N+ \frac{(m-1)m+(n-1)n}{\beta}+\frac{m+n}{2}}  \nonumber \\
  && \times N^{\frac{(m+n)(m+n-1)}{\beta}+\frac{m+n}{2}} \prod_{j=0}^{m+n-1}\frac{\Gamma(1+2/\beta)}{\Gamma(1+2(1+j)/\beta)}.   \label{CNSN}
\end{eqnarray}

%

Finally, combine the asymtotics of $I_{N}$ with the above and  also notice \eqref{eq1}, \eqref{sind}, \eqref{gaussint} and \eqref{vip2},  we thus complete the proof.
\end{proof}

When $\beta$ is an even integer, the scaling limit of correlation functions immediately follows from Theorem \ref{bulklimits}.
 \begin{coro}[Bulk correlations for even $\beta$]\label{cfbulklimits} Let   $$\rho=\frac{1}{2\pi}\frac{\sqrt{\sin^{2}(\theta/2)-\sin^{2}(\theta_\textrm{d}/2)}}{(1-\sin(\theta_\textrm{d}/2))\sin(\theta/2)}, \qquad \theta_\textrm{d}<\theta<2\pi-\theta_\textrm{d},$$   and introduce the scaled variables  $$r_{j}= \theta+\frac{y_{j}}{\rho N} ,\quad    j=1,\ldots,k$$
with real variables  $y_{1}, \ldots, y_{k}$ in \eqref{defpcf}.   Assume that $\beta$ is even and $b=\beta Nd/2$ with $d>0$, then    \begin{multline}
\lim_{N\rightarrow \infty} \big(\frac{1}{\rho N}\big)^{k} R^{(k)}_{b,N}\big(\theta+\frac{y_{1}}{\rho N},\ldots,\theta+\frac{y_{k}}{\rho N}\big)=
c_{k}(0,\beta) |\Delta_{k}(2\pi y)|^{\beta} \\ \, \times e^{-i\beta\pi \sum_{j=1}^{k} y_{j}}\,
\!{\phantom{j}}_1F_1^{(\beta/2)}(k;2k;  2i\pi y_{1},\ldots, 2i\pi y_{1}, \ldots, 2i\pi y_{k}, \ldots,  2i\pi y_{k})
\end{multline}
where in the argument of ${\phantom{j}}_1F_1^{(\beta/2)}$ each $2i\pi y_{j}$ ($j=1,\ldots,k$) occurs $\beta$ times, and $c_{k}(0,\beta)$ is defined in \eqref{universalconstantcfbulk}.
\end{coro}

\begin{proof}Let $m=\beta k/2$. Introduce $s_j=e^{i(2\pi-\theta)+i \frac{x_j}{\rho N}}, \ j=1,\ldots,m$ and
$$x_{l+ \beta(j-1)/2}=-y_j, \quad j=1,\ldots,k, l=1,\ldots, \beta/2,$$
by  use of the formula \eqref{cptopcf}  as $N\rightarrow \infty$ we get
\begin{multline}
 \big(\frac{1}{\rho N}\big)^{k} R^{(k)}_{b,N}\big(\theta+\frac{y_{1}}{\rho N},\ldots,\theta+\frac{y_{k}}{\rho N}\big)\sim \frac{ M_{N}(b,b;\beta/2)}{M_{k+N}(b,b;\beta/2)}\, \rho^{-k-\beta k(k-1)/2}\times\\
 N^{-\beta k(k-1)/2} (2\sin\frac{\theta}{2})^{\beta d kN} \exp\big\{  \frac{\beta d}{2\rho}
  \cot\frac{\theta}{2}  \sum_{j=1}^k   y_{j} \big\}  \ |\Delta_{k}(y)|^{\beta} \,
K_{b,N}(s;s).
\end{multline}

Substitute $\theta$ by $2\pi-\theta$ in Theorem \ref{bulklimits}, and notice the following asymptotics (cf. the formula \eqref{Mconstant})
 \begin{multline}\frac{ M_{N}(b,b;\beta/2)}{M_{k+N}(b,b;\beta/2)}\sim (2\pi)^{-k}\big(\Gamma(1+\beta/2)\big)^{k} (\beta N/2)^{-\beta k/2}\,\times\\
 (1+d)^{\beta k(k-1)/2+\beta k(1+d)N+k}(1+2d)^{-\beta k(k-1)/4-\beta k(1+2d)N/2-k/2} \label{Mratio}
 ,\end{multline}
 we have
 \begin{multline}
 \big(\frac{1}{\rho N}\big)^{k} R^{(k)}_{b,N}\big(\theta+\frac{y_{1}}{\rho N},\ldots,\theta+\frac{y_{k}}{\rho N}\big)\sim
 (2\pi)^{\beta k(k-1)/2}(\beta/2)^{-\beta k/2}\times \\
 \big(\Gamma(1+\beta/2)\big)^{k} \gamma_{m}(4/\beta) \ |\Delta_{k}(y)|^{\beta} \,   \!{\phantom{j}}_0\mathcal{F}_0^{(\alpha)}(i\pi x;1^{(m)},(-1)^{(m)})
 .
\end{multline}

Finally, use the Gauss's multiplication formulas of the gamma functions (cf. the equation \eqref{gaussmul}) for $\gamma_{m}(4/\beta)$ and we thus complete the proof.
\end{proof}

\subsection{Soft edge limits}

The asymptotic results at the soft edge  will be written in terms of  multivariate Airy functions of $x_1, \ldots, x_n$  \cite{des,dl} defined by
\be\label{Airydef} \mathrm{Ai}_{n}^{(\alpha)}(x)=\frac{1}{(2\pi)^n}\int_{\mathbb{R}^n}e^{i\sum_{j=1}^{n}w_j^{3}/3}\ |\Delta_{n}(w)|^{2/\alpha}{\phantom{j}}_0\mathcal{F}_0^{(\alpha)}(x;i w)\, d^{n}w.   \ee
Note that, by convention,   the  notation $\pm$   refers  to the sign of the involved quantities  corresponding to the spectral edge     $\theta= \theta_\textrm{d}$ and $\theta= 2\pi-\theta_\textrm{d}$, respectively.

\begin{theorem}[Soft edge limit]\label{edgelimits} Let   $$\rho=\pm(1-\sin\frac{\theta_\textrm{d}}{2} )^{-2/3} (-\cot\frac{\theta_\textrm{d}}{2})^{1/3} (4N)^{-1/3},$$   and introduce the scaled variables $$s_{j}=e^{i\theta+i\frac{x_{j}}{\rho N}}, \ j=1,\ldots,m \quad \mbox{and} \quad t_{k}=e^{i\theta+i\frac{x_{m+k}}{\rho N}}, \  k=1,\ldots,n$$
where $\theta=\theta_\textrm{d}$ or $\theta=2\pi-\theta_\textrm{d}$. Assume that   $b=\beta Nd/2$ with $d>0$, then as $N\rightarrow \infty$ we have \begin{multline}
   \frac{K_{b,N}(s;t)}{\Phi_{N,m,n}}\sim
\exp\Big\{ \frac{1}{2\rho}
\big[
(i\pm  \frac{-\cos\frac{\theta_\textrm{d}}{2}}{1-\sin\frac{\theta_\textrm{d}}{2}}  ) \sum_{j=1}^m   x_{j}
 -(i\pm \frac{\cos\frac{\theta_\textrm{d}}{2}}{1-\sin\frac{\theta_\textrm{d}}{2}}  ) \sum_{j=1}^n   x_{m+j}
\big]
\Big\}  \\ \, \times (2\pi)^{m+n} (\Gamma_{4/\beta,m+n})^{-1} \mathrm{Ai}_{m+n}^{(\beta/2)}(x_{1},\ldots,x_{m+n}).
\end{multline}
Here $K_{b,N}(s;t)$ is defined in \eqref{productmean}, $\Gamma_{4/\beta,m+n}$ is given in \eqref{gaussint} and \begin{align}&\Phi_{N,m,n}=\nonumber\\
&2^{-(m+n)dN+\frac{1}{3}(m+n)-\frac{1}{3\beta}(m+n)(m+n+2)}  d^{-(m+n)dN+\frac{1}{3\beta}(m^2+n^2-4mn-m-n)}\nonumber\\
&\times (1+d)^{-(m+n)N+\frac{1}{6}(m+n)-\frac{1}{3\beta}(2m^2+2n^2-2mn+m+n)} \nonumber\\
&\times (1+2d)^{-\frac{1}{2}(m+n)(1+2d)N-\frac{1}{3}(m+n)+\frac{1}{3\beta}(m+n)(m+n+2)} \nonumber\\
 &\times \big(\frac{\sqrt{1+2d}-i}{\sqrt{1+2d}+i}\big)^{\pm(m-n)(m+n)/\beta}
 e^{\pm i(m-n)N\theta_\textrm{d}/2} N^{\frac{1}{6}(m+n)+\frac{1}{3\beta}(m+n)(m+n-1)}.
\end{align}
\end{theorem}
\begin{proof} As in the bulk case we now continue the foregoing analysis in Section \ref{general}. But this time we will make use of Corollary 3.11   in \cite{dl}, which deals with one saddle point of higher order. Since the two cases of $\theta= \theta_\textrm{d}$ and $\theta= 2\pi-\theta_\textrm{d}$ are similar, we only consider the former.

First, at $\theta= \theta_\textrm{d}$    simple manipulation from \eqref{pfunction} and \eqref{one+} shows
\begin{align}\label{ponevalue} p(x_{0})= -i\frac{\theta_\textrm{d}}{2}+\log\frac{1+\sin\frac{\theta_\textrm{d}}{2}}{\cos\frac{\theta_\textrm{d}}{2}} +d\log2 +d\log\big(1+\sin\frac{\theta_\textrm{d}}{2}\big),  \end{align}
and \begin{align} \label{p3prime}p&'''(x_{0})= \frac{
8i(1+\sin\frac{\theta_\textrm{d}}{2})^{3}
}{
(1-\sin\frac{\theta_\textrm{d}}{2})\cot\frac{\theta_{\textrm{d}}}{2}
}.  \end{align}
On the other hand,
\be \frac{ie^{i\theta} x_0}{1-u^{-1}x_0}=\frac{i}{2}-\frac{ \cos\frac{\theta_\textrm{d}}{2} }{2(1-\sin\frac{\theta_\textrm{d}}{2})},\ee
 thus by \eqref{vip1} we can rewrite \eqref{eq2} as
  \begin{align}  I_{N}&=\exp\Big\{ \frac{1}{2\rho}
 \big(i-  \frac{\cos\frac{\theta_\textrm{d}}{2}}{1-\sin\frac{\theta_\textrm{d}}{2}}  \big)  \sum_{j=1}^{m+n}   x_{j}
 \Big\}\nonumber \\
 \times &\int_{[0,1]^{m+n}} \exp\{-N\sum_{j=1}^{m+n}p(y_j)\} \,|\Delta_{m+n}(y)|^{2/\alpha}\, q(y) g(y)\,d^{m+n}y
 \end{align}
 where
  \be   q(y)=\prod_{j=1}^{m+n}y_{j}^{ \frac{1-m}{\alpha}-1}(1-y_{j})^{ \frac{1-n}{\alpha}-1 },  \ee
  and \be \label{softg(y)} g(y)={\!\!{\phantom{j}}_0\mathcal{F}_{0}^{(\alpha)}(x/\rho;\frac{ie^{i\theta}y}{1-u^{-1}y}- \frac{ie^{i\theta} x_0}{1-u^{-1}x_0})}.\ee

 Secondly, notice  the expansion at $x_0$
 \be \frac{ie^{i\theta} y_j}{1-u^{-1}y_j}-\frac{ie^{i\theta} x_0}{1-u^{-1}x_0}= \frac{i(1+\sin\frac{\theta_\textrm{d}}{2})}{1-\sin\frac{\theta_\textrm{d}}{2} }(y_j-x_0)+\cdots, \nonumber
 \ee
 let $$\rho=(1-\sin\frac{\theta_\textrm{d}}{2} )^{-2/3} (-\cot\frac{\theta_\textrm{d}}{2})^{1/3} (4N)^{-1/3},$$
 then by   Corollary 3.11 of \cite{dl} (the situation  is a little different,  but the proof there can be easily modified to be available for the form of \eqref{softg(y)}) we get
  \begin{align}
&I_{N} \sim   \exp\Big\{ \frac{1}{2\rho}
 \big(i-  \frac{\cos\frac{\theta_\textrm{d}}{2}}{1-\sin\frac{\theta_\textrm{d}}{2}}  \big)  \sum_{j=1}^{m+n}   x_{j}
 \Big\} \frac{e^{-(m+n)N p(x_0)}}{N^{\frac{(m+n)(m+n-1)}{3\alpha}+\frac{m+n}{3}}} \, q(x_{0}, \dots, x_{0})\nonumber \\
 &\times \int_{\mathbb{R}^{m+n}}e^{-\frac{p'''(x_0)}{6}\sum_{j=1}^{n} w_j^{3}}\ |\Delta_{n}(w)|^{2/\alpha}{\phantom{j}}_0\mathcal{F}_0^{(\alpha)}(\frac{x}{\rho N^{1/3}};  \frac{i(1+\sin\frac{\theta_\textrm{d}}{2})}{1-\sin\frac{\theta_\textrm{d}}{2} }w)\, d^{m+n}w.\nonumber
 \end{align}
Substitution of   \eqref{ponevalue} and \eqref{p3prime}, it could simplify to
  \begin{align}
&I_{N}\sim   N^{-\frac{(m+n)(m+n-1)}{3\alpha}-\frac{m+n}{3}}
  \exp\Big\{ \frac{1}{2\rho}
 \big(i-  \frac{\cos\frac{\theta_\textrm{d}}{2}}{1-\sin\frac{\theta_\textrm{d}}{2}}  \big)  \sum_{j=1}^{m+n}   x_{j}
 \Big\}     \nonumber \\
 &  \times \exp\Big\{ -(m+n)N  \big(-i \frac{\theta_\textrm{d}}{2}+\log\frac{1+\sin\frac{\theta_\textrm{d}}{2}}{ \cos\frac{\theta_\textrm{d}}{2}} +d \log2+d\log(1+\sin\frac{\theta_\textrm{d}}{2})\big)\Big\} \nonumber\\
 & \times \Big(\frac{4(1+\sin\frac{\theta_\textrm{d}}{2})^{3}}{ (1-\sin\frac{\theta_\textrm{d}}{2})\cot\frac{\theta_\textrm{d}}{2}}\Big)^{-\frac{(m+n)(m+n-1)}{3\alpha}-\frac{m+n}{3}}
   \Big(\frac{1}{2}+\frac{i\cos\frac{\theta_\textrm{d}}{2}}{2(1+\sin\frac{\theta_\textrm{d}}{2})} \Big)^{(m+n)(\frac{1-m}{\alpha} -1)}\nonumber\\
   & \times \Big(\frac{1}{2}-\frac{i\cos\frac{\theta_\textrm{d}}{2}}{2(1+\sin\frac{\theta_\textrm{d}}{2})} \Big)^{(m+n)(\frac{1-n}{\alpha} -1)} (2\pi)^{m+n} \mathrm{Ai}_{m+n}^{(\alpha)}(x_{1},\ldots,x_{m+n}).\label{finalasym}
\end{align}

Finally, combine  \eqref{eq1}, \eqref{sind}, \eqref{CNSN} and \eqref{finalasym},  we thus complete the proof after careful calculation.
\end{proof}

When $\beta$ is an even integer, the scaling edge limit of correlation functions immediately follows from Theorem \ref{edgelimits}.
 \begin{coro}[Soft edge  correlations for even $\beta$]\label{cfedgelimits} Let
 $$\rho=\pm(1-\sin\frac{\theta_\textrm{d}}{2} )^{-2/3} (-\cot\frac{\theta_\textrm{d}}{2})^{1/3} (4N)^{-1/3},$$
 and introduce the scaled variables  $$r_{j}= \theta+\frac{y_{j}}{\rho N} ,\quad    j=1,\ldots,k$$
in \eqref{defpcf}.     Assume that  $\beta$ is even and $b=\beta Nd/2$ with $d>0$, then for   $\theta=\theta_\textrm{d}$ or $\theta=2\pi-\theta_\textrm{d}$ we have
  \begin{multline}
\lim_{N\rightarrow \infty} \big(\frac{1}{|\rho| N}\big)^{k} R^{(k)}_{b,N}\big(\theta+\frac{y_{1}}{\rho N},\ldots,\theta+\frac{y_{k}}{\rho N}\big) \\ =
a_{k}(\beta) |\Delta_{k}( y)|^{\beta} \mathrm{Ai}_{\beta k}^{(\beta/2)}(y_{1},\ldots, y_{1},\ldots, y_k, \ldots,  y_{k})
\end{multline}
where in the argument of $\mathrm{Ai}_{\beta k}^{(\beta/2)}$ each $y_{j}$ ($j=1,\ldots,k$) occurs $\beta$ times, and
\be a_{k}(\beta)=  (\beta/2)^{(\beta k+1)k}(\Gamma(1+\beta/2))^{k} \prod_{j=1}^{2k} \frac{(\Gamma(1+2/\beta))^{\beta /2}}{\Gamma(1+\beta j/2) }\label{univcoefficienta}. \ee
\end{coro}

\begin{proof}Let $m=\beta k/2$. Introduce $s_j=e^{i(2\pi-\theta)+i \frac{x_j}{\rho N}}, \ j=1,\ldots,m$ and
$$x_{l+ \beta(j-1)/2}=-y_j, \quad j=1,\ldots,k, l=1,\ldots, \beta/2,$$
by  use of the formula \eqref{cptopcf}  as $N\rightarrow \infty$ we get
\begin{multline}
 \big(\frac{1}{|\rho| N}\big)^{k} R^{(k)}_{b,N}\big(\theta+\frac{y_{1}}{\rho N},\ldots,\theta+\frac{y_{k}}{\rho N}\big)\sim \frac{ M_{N}(b,b;\beta/2)}{M_{k+N}(b,b;\beta/2)}\, |\rho|^{-k-\beta k(k-1)/2}\times\\
 N^{-\beta k(k-1)/2} (2\sin\frac{\theta}{2})^{\beta d kN} \exp\big\{  \frac{\beta d}{2\rho}
  \cot\frac{\theta}{2}  \sum_{j=1}^k   y_{j} \big\}  \ |\Delta_{k}(y)|^{\beta} \,
K_{b,N}(s;s).
\end{multline}

Substitute $\theta$ by $2\pi-\theta$ in Theorem \ref{edgelimits}, and notice the   asymptotics of \eqref{Mratio},  application of Theorem \ref{edgelimits} after some  cumbersome calculations and clever cancellation    gives
  \begin{multline}
 \big(\frac{1}{\rho N}\big)^{k} R^{(k)}_{b,N}\big(\theta+\frac{y_{1}}{\rho N},\ldots,\theta+\frac{y_{k}}{\rho N}\big)\sim
 (2\pi)^{m-k}(\beta/2)^{-\beta k/2}\big(\Gamma(1+\beta/2)\big)^{k} \\
  \times \prod_{j=1}^{2m} \frac{(\Gamma(1+2/\beta))^{\beta /2}}{\Gamma(1+2j/\beta) } \ |\Delta_{k}(y)|^{\beta} \,   \mathrm{Ai}_{\beta k}^{(\beta/2)}(y_{1},\ldots, y_{1},\ldots, y_k, \ldots,  y_{k}).
\end{multline}

Finally, use the Gauss's multiplication formulas of the gamma functions (cf. the equation \eqref{gaussmul}) for the product $$\prod_{j=1}^{2m} \frac{(\Gamma(1+2/\beta))^{\beta /2}}{\Gamma(1+2j/\beta) }$$ and we thus complete the proof.
\end{proof}

\section{Spectrum singularity-to-soft edge transition}

In this section   we describe the transition from the limiting multivariate function near the spectrum singularity  to
the multivariate Airy function at the soft edge as $b$ goes to infinity. Similar results also hold for correlation functions when $\beta$ is even.

We first point out some obvious properties of the function  $S_{b}^{(\alpha)}(x)$ defined by \eqref{ssdef} which describes  the scaling limit near the spectrum singularity.
 With $b=0$ the circular Jacobi ensemble  corresponds to Dyson's circular ensemble, thus we have proved that $S_{0}^{(\alpha)}(x)$ is actually  the bulk limit for the latter.
 Particularly when $m=n$ it is locally identical to the bulk limit of the Hermite (Gaussian), Laguerre (Chiral) and
 Jacobi  ensembles proved in \cite{dl}, see also Theorem \ref{bulklimits}.  Another relationship between \eqref{ssdef} and the bulk limit  follows from the more general identity
 \begin{align}
  S_{b}^{(\alpha)}(x_1,\ldots,x_{m+n},0,0)= \frac{\Gamma(\frac{1+ \bar{b}}{\alpha})\,\Gamma(\frac{1+ b}{\alpha}) }
{\Gamma(\frac{1+ \bar{b}+b}{\alpha})\, \Gamma(\frac{2+ \bar{b}+b}{\alpha})}
  S_{b+1}^{(\alpha)}(x_1,\ldots,x_{m+n}),\end{align}
which can be easily derived from the fact $$P_{\kappa}^{(\alpha)}(x_1,\ldots,x_n,0)=P_{\kappa}^{(\alpha)}(x_1,\ldots,x_n).$$

We now state the transition results as follows.
 \begin{theorem}\label{transition} Assume that $b\in \mathbb{R}$.
 Introduce the scaled variables $$x_{j}=(2b/\alpha)-( 4b/\alpha)^{1/3}\eta_{j} , \ j=1,\ldots,m+n,$$
 as $b\rightarrow \infty$ we have
  \begin{align} &S_{b}^{(\alpha)}(x)\sim  
 (-1)^{m+n} 2^{-\frac{b}{\alpha}(m+n)+ \frac{n-2m-1}{3\alpha}(m+n)+\frac{m+n}{3}}(1-i)^{\frac{m^2-n^2}{\alpha}}e^{\frac{1}{2}(\frac{4b}{\alpha})^{1/3}\sum_{j=1}^{m+n} \eta_j}\nonumber\\
 &(b/\alpha)^{\frac{m^2+n^2-4mn-m-n}{6\alpha}-\frac{b(m+n)}{\alpha}+\frac{m+n}{6}}(2\pi)^{m+n}(\Gamma_{2/\alpha,m+n})^{-1}  \mathrm{Ai}_{m+n}^{(\alpha)}(\eta)
 \end{align}
  where   $\Gamma_{2/\alpha,m+n}$ is given in \eqref{gaussint}.  
\end{theorem}

\begin{proof}
 For $\textrm{Re}\{b\}>\max\{m,n\}-1$, by the  integral  representation due to Yan \cite{yan}  we have \begin{align}
S_{b}^{(\alpha)}(x)&
 = \frac{ \gamma_{m,n}(b,2/\alpha)}{S_{m+n}(\nu_1,\nu_2;1/\alpha)}\, e^{-\frac{i}{2} \sum_{j=1}^{m+n} x_j}\times\nonumber\\
&\int_{[0,1]^{m+n}} {\!\!\!{\phantom{j}}_0\mathcal{F}_{0}^{(\alpha)}(ix;y)}\prod_{j=1}^{m+n}y_{j}^{\nu_1}(1-y_{j})^{\nu_2}\,|\Delta_{m+n}(y)|^{2/\alpha}\,d^{m+n}y,
\end{align}
where $\nu_{1}= -1+(b-m+1)/\alpha, \, \nu_{2}= -1+(b-n+1)/\alpha$ and $S_{m+n}(\nu_1,\nu_2;1/\alpha)$ is defined as in \eqref{selbergc}.

With $x_{j}=(2b/\alpha)-( 4b/\alpha)^{1/3}\eta_{j}$   application  of \eqref{vip1} gives
\begin{align}
&S_{b}^{(\alpha)}(x)
 = \frac{ \gamma_{m,n}(b,2/\alpha)}{S_{m+n}(\nu_1,\nu_2;1/\alpha)}\, e^{-\frac{1}{2} \sum_{j=1}^{m+n} x_j} \int_{[0,1]^{m+n}} \prod_{j=1}^{m+n}y_{j}^{\frac{1-m}{\alpha}-1}(1-y_{j})^{\frac{1-n}{\alpha}-1}
\nonumber\\
& \times e^{-\frac{b}{\alpha}\sum_{j=1}^{m+n}p(y_j)} {\!\!\!{\phantom{j}}_0\mathcal{F}_{0}^{(\alpha)}\big(-i(\tfrac{4b}{\alpha})^{1/3}\eta;y-\tfrac{1+i}{2}\big)}
\,|\Delta_{m+n}(y)|^{2/\alpha}\,d^{m+n}y,
\end{align}
where \be p(z)=-2i(z-\tfrac{1+i}{2})-\log z-\log(1-z).\label{pz}\ee
 We thus derive  from
 \be p'(z)=-2i - \frac{1}{z}+\frac{1}{1-z}=0\ee
 that the saddle point $z_0=\tfrac{1+i}{2}$. Obviously, $p''(z_0)=0$, $p'''(z_0)=8i$ and $p(z_0)=\log 2$.

 Let's deform the interval  $[0,1]$ to the contour $\mathcal{C}$ consisting of three straight line paths from
 $(0,0)$ to $(0,i/2)$, then from $(0,i/2)$ to $(1, i/2)$ and then from $(1,i/2)$ to $(1,0)$. It is easy to check that $\textrm{Re}\{p(z)\}$
  attains its global minimum at $z=z_0$. Hence according to Corollary 3.11 \cite{dl}    as $b\to\infty$ we have
  \begin{align}
&S_{b}^{(\alpha)}(x)
\sim \frac{ \gamma_{m,n}(b,2/\alpha)}{S_{m+n}(\nu_1,\nu_2;1/\alpha)}\, e^{-\frac{1}{2} \sum_{j=1}^{m+n} x_j}  \, e^{-\frac{b}{\alpha} (m+n)p(z_0)}
\big(z_{0}^{\frac{1-m}{\alpha}-1}(1-z_{0})^{\frac{1-n}{\alpha}-1}\big)^{m+n}
\nonumber\\
& \times \int_{\mathds{R}^{m+n}}  e^{-\frac{b}{\alpha}\frac{p'''(z_0)}{6}\sum_{j=1}^{m+n}w_{j}^{3}} {\!\!\!{\phantom{j}}_0\mathcal{F}_{0}^{(\alpha)}\big(-i(\tfrac{4b}{\alpha})^{1/3}\eta;w\big)}
\,|\Delta_{m+n}(w)|^{2/\alpha}\,d^{m+n}w \nonumber\\
&=\frac{ \gamma_{m,n}(b,2/\alpha)}{S_{m+n}(\nu_1,\nu_2;1/\alpha)}
\, e^{-\frac{1}{2} \sum_{j=1}^{m+n} x_j} (-1)^{m+n}
2^{-\frac{b}{\alpha}(m+n)+ (\frac{n-1}{\alpha}+1)(m+n)}
(1-i)^{\frac{m^2-n^2}{\alpha}}\nonumber\\
&\times (4b/\alpha)^{-\frac{(m+n)(m+n-1)}{3\alpha}-\frac{m+n}{3}}(2\pi)^{m+n} \mathrm{Ai}_{m+n}^{(\alpha)}(\eta).\label{S1asymptotics}
\end{align}

Using the Stirling formula we can  derive
\begin{align}
  \frac{ \gamma_{m,n}(b,2/\alpha)}{S_{m+n}(\nu_1,\nu_2;1/\alpha)}&\sim (2\pi)^{-\frac{m+n}{2}}
\prod_{j=1}^{m+n}\frac{\Gamma(1+\frac{1}{\alpha})}{\Gamma(1+\frac{j}{\alpha})}  \nonumber\\
&\times \, e^{-\frac{b}{\alpha} (m+n)}
 (b/\alpha)^{\frac{ m}{2\alpha}(m-1-2b)+\frac{n}{2\alpha}(n-1-2b)+\frac{m+n}{2}}.
\end{align}
Together with \eqref{S1asymptotics} we thus complete the proof.
\end{proof}

For  $\beta$ even, a transition  from   the   limiting correlation functions near the spectrum singularity to the soft-edge correlation
 immediately follows from Theorem \ref{transition}.
 \begin{coro} Assume that $\beta$ is   even and $b\in \mathbb{R}$. Let $R^{(k,\beta/2)}_{b,1}(y)$  be defined by \eqref{bcorrelation}, we have    \begin{multline}
\lim_{b\rightarrow \infty} \big(\tfrac{8b}{\beta}\big)^{k/3}R^{(k,\beta/2)}_{b,1}\big(-\tfrac{4b}{\beta}+(\tfrac{8b}{\beta})^{1/3}\xi_1,\ldots,-\tfrac{4b}{\beta}+(\tfrac{8b}{\beta})^{1/3}\xi_k\big)
=\\
a_{k}(\beta) |\Delta_{k}(\xi)|^{\beta} \mathrm{Ai}_{\beta k}^{(\beta/2)}(\xi_{1},\ldots, \xi_{1},\ldots, \xi_k, \ldots,  \xi_{k})
\end{multline}
where in the argument of $\mathrm{Ai}_{\beta k}^{(\beta/2)}$ each $\xi_{j}$ ($j=1,\ldots,k$) occurs $\beta$ times, and
 $a_{k}(\beta)$ is defined by \eqref{univcoefficienta}.
\end{coro}
\begin{proof} Let $\alpha=\beta/2$. Combine the definition  of $S^{(\alpha)}_b(x)$ and Theorem \ref{transition}, one easily completes the proof.
\end{proof}
\begin{acknow}
The author thanks PJ Forrester for bringing reference \cite{jok} to his attention and for very helpful comments, and also thanks P Desrosiers for useful discussion during
his stay in Chile.
The work  was  supported by the National Natural Science Foundation of China (Grants  11301499 and 11171005) and by CUSF WK 0010000026.
\end{acknow}

\end{document}